\documentclass[11pt, a4paper, german]{article}

\usepackage[english]{babel}
\usepackage{hyphsubst}
\usepackage{mathtools}
\usepackage{amssymb}
\usepackage{amsthm}
\usepackage{graphics} 
\usepackage{url}
\usepackage{hyperref}

\usepackage{graphicx}
\usepackage{caption}
\usepackage{subcaption}
\usepackage{color} 

\usepackage{ wasysym }

\usepackage{tikz}
\usepackage{amssymb,amsfonts,amsmath,MnSymbol}
\numberwithin{equation}{section}

\usepackage{tkz-euclide}
\usepackage{circuitikz}
\usepackage{xcolor}
\usetikzlibrary{arrows,calc,patterns}
\definecolor{hellgelb}{RGB}{255,255,170}
\definecolor{tuerkis}{RGB}{0,219,159}
\definecolor{hellgrau}{RGB}{141,141,141}
\definecolor{petrol}{RGB}{0,152,240}
\definecolor{neutrales_blau}{RGB}{0,102,183}
\definecolor{koralle}{RGB}{255,71,71}
\definecolor{gruen}{RGB}{37,179,40}
\definecolor{blasses_blau}{RGB}{84,155,226}

\DeclareFontShape{OT1}{cmss}{b}{n}{<->ssub * cmss/bx/n}{}
\usepackage{upgreek}

\usepackage[nottoc]{tocbibind}

\usepackage{xfrac}

\usepackage{xcolor}

\usepackage{amsthm}

\usepackage{float}

\usepackage{enumerate}

\usepackage[a4paper]{geometry}

\author{{Joscha Gedicke \thanks{Institut f\"ur Numerische Simulation, Universit\"at Bonn, 53115 Bonn (gedicke@ins.uni-bonn.de)} \quad and \quad Karen Petersen \thanks{Institut f\"ur Numerische Simulation, Universit\"at Bonn, 53115 Bonn (petersen@ins.uni-bonn.de)}
}}

\date{}

\title{Partition of unity a posteriori error estimators for the linear elasticity eigenvalue problem}

\begin{document}


\newtheorem{theorem}{Theorem}[section]
\newtheorem{corollary}[theorem]{Corollary}
\newtheorem{lemma}[theorem]{Lemma}
\newtheorem{definition}[theorem]{Definition}
\newtheorem{proposition}[theorem]{Proposition}
\newtheorem{example}[theorem]{Example}
\newtheorem*{remark}{Remark}
\newtheorem{problem}[theorem]{Problem}
\newtheorem{weak_formulation}[theorem]{Weak Formulation}
\newtheorem{discrete_problem}[theorem]{Discrete Problem}

\newcommand{\R}{\mathbb{R}}
\newcommand{\C}{\mathbb{C}}
\newcommand{\arund}{\mathnormal{a}}
\newcommand{\arundz}{\arund_z}
\newcommand{\aetaz}{\arund_{\etazew}}
\newcommand{\amuz}{\arund_{\muzew}}
\newcommand{\rot}{\text{\normalfont rot}\,}
\newcommand{\ufett}{\mathbf{u}}
\newcommand{\efett}{\mathbf{e}}
\newcommand{\taufett}{\text{\boldmath$\tau$}}
\newcommand{\phifett}{\text{\boldmath$\phi$}}
\newcommand{\nufett}{\text{\boldmath$\nu$}}
\newcommand{\ufetth}{\ufett_h}
\newcommand{\Hfett}{\mathbf{H}}
\newcommand{\Pcal}{\mathcal{P}}
\newcommand{\Pcalkl}{\Pcal_{k+\ell}}
\newcommand{\Pfett}{\mathbf{P}}
\newcommand{\xfett}{\mathbf{x}}
\newcommand{\vfett}{\mathbf{v}}
\newcommand{\vfetth}{\mathbf{v}_h}
\newcommand{\rfett}{\mathbf{r}}
\newcommand{\Efett}{\mathbf{E}}
\newcommand{\Efettj}{\Efett^{(j)}}
\newcommand{\wfett}{\mathbf{w}}
\newcommand{\Wfett}{\mathbf{W}}
\newcommand{\Wfettzl}{\Wfett_z^{\ell}}
\newcommand{\ffett}{\mathbf{f}}
\newcommand{\nfett}{\mathbf{n}}
\newcommand{\tfett}{\mathbf{t}}
\newcommand{\Vfett}{\mathbf{V}}
\newcommand{\Vfetth}{\Vfett_h}
\newcommand{\Vfettzl}{\Vfett_z^{\ell}}
\newcommand{\Lfett}{\mathbf{L}}
\newcommand{\nablafett}{\mathbf{\nabla}}
\newcommand{\straintensor}{\varepsilon(\ufett)}
\newcommand{\domain}{\Omega \subseteq \R^n}
\newcommand{\domainrtwo}{\Omega \subseteq \R^2}
\newcommand{\sigmau}{\text{\boldmath$\sigma$}}
\newcommand{\sigmauh}{\sigmau_h}
\newcommand{\divsigmau}{\mathbf{div}\,\sigmau}
\newcommand{\divsigmauh}{\mathbf{div}\,\sigmauh}
\newcommand{\divcepsu}{\mathbf{div} (\C \, \epsu)}
\newcommand{\divcepsuh}{\mathbf{div} (\C \, \epsuh)}
\newcommand{\eps}{\varepsilon}
\newcommand{\epsu}{\varepsilon(\ufett)}
\newcommand{\Cepsu}{\C \, \epsu}
\newcommand{\epsv}{\varepsilon(\vfett)}
\newcommand{\epsw}{\varepsilon(\wfett)}
\newcommand{\epsuh}{\varepsilon(\ufetth)}
\newcommand{\Cepsuh}{\C \, \epsuh}
\newcommand{\epsvh}{\varepsilon(\vfetth)}
\newcommand{\epse}{\varepsilon(\efett)}
\newcommand{\lambdak}{\lambda^{(k)}}
\newcommand{\kappaj}{\kappa^{(j)}}
\newcommand{\kappai}{\kappa^{(i)}}
\newcommand{\kappah}{\kappa_h}
\newcommand{\kappajh}{\kappaj_h}
\newcommand{\ufettj}{\ufett^{(j)}}
\newcommand{\ufetti}{\ufett^{(i)}}
\newcommand{\ufettjh}{\ufettj_h}
\newcommand{\Ttilde}{\tilde{T}}
\newcommand{\TaufH}{\Ttilde|_{\Hfett}}
\newcommand{\Tbar}{\bar{T}}
\newcommand{\TbarH}{\Tbar|_{\Hfett}}
\newcommand{\traceop}{\gamma_{\nfett}}
\newcommand{\Hfettn}{\Hfett_{\nfett}}
\newcommand{\phiz}{\varphi_z}
\newcommand{\Tcal}{\mathcal{T}}
\newcommand{\Ecal}{\mathcal{E}}
\newcommand{\Ncal}{\mathcal{N}}
\newcommand{\tr}{\text{tr}}
\newcommand{\bilform}{\arund(\cdot,\cdot)}
\newcommand{\bilformtilde}{\atilde(\cdot,\cdot)}
\newcommand{\bilformtildez}{\atildez(\cdot,\cdot)}
\newcommand{\bilformbar}{\bar{\arund}(\cdot,\cdot)}
\newcommand{\clement}{\mathtt{I}_h}
\newcommand{\CKorn}{C_{\operatorname{Kornelast}}}
\newcommand{\CWKorn}{C_{\operatorname{WKorn}}}
\newcommand{\CPF}{C_{\operatorname{PF}}}
\newcommand{\RBM}{\operatorname{RM}}
\newcommand{\Young}{\operatorname{Young}}
\newcommand{\CS}{\operatorname{C.S.}}

\newcommand{\GammaD}{\Gamma_D}
\newcommand{\Gammag}{\Gamma_g}
\newcommand{\GammaN}{\Gamma_N}

\newcommand{\intomega}{\int_{\Omega} \,}
\newcommand{\intrandomega}{\int_{\partial \Omega} \,}
\newcommand{\intrandomegaD}{\int_{\GammaD} \,}
\newcommand{\intrandomegag}{\int_{\Gammag} \,}
\newcommand{\intrandomegaN}{\int_{\GammaN} \,}

\newcommand{\mathbbm}[1]{\text{\usefont{U}{bbm}{m}{n}#1}}

\newcommand{\vectorxy}{\begin{pmatrix} x \\ y  \end{pmatrix}}
\newcommand{\vectorrot}{\begin{pmatrix} -y \\ x  \end{pmatrix}}
\newcommand{\eoneintwod}{\begin{pmatrix} 1 \\ 0  \end{pmatrix}}
\newcommand{\etwointwod}{\begin{pmatrix} 0 \\ 1  \end{pmatrix}}
\newcommand{\nullfett}{\mathbf{0}}
\newcommand{\einsfett}{\mathbf{1}}
\newcommand{\vectorxyz}{\begin{pmatrix} x \\ y \\ z \end{pmatrix}}
\newcommand{\eoneinthreed}{\begin{pmatrix} 1 \\ 0 \\ 0 \end{pmatrix}}
\newcommand{\etwointhreed}{\begin{pmatrix} 0 \\ 1 \\ 0 \end{pmatrix}}
\newcommand{\ethreeinthreed}{\begin{pmatrix} 0 \\ 0 \\ 1 \end{pmatrix}}
\newcommand{\into}{\hookrightarrow}

\newcommand{\etalew}{\eta}
\newcommand{\mulew}{\mu}
\newcommand{\etazew}{\eta_z}
\newcommand{\etazewh}{\eta_{z,h}}
\newcommand{\muzew}{\mu_z}
\newcommand{\muzewh}{\mu_{z,h}}
\newcommand{\etares}{\eta_{res}}
\newcommand{\Azeta}{A_z}
\newcommand{\Azmu}{\tilde{A}_z}

\maketitle

\begingroup
\renewcommand{\thefootnote}{}
\footnotetext{The authors gratefully acknowledge the granted access to the Bonna cluster hosted by the University of Bonn.}
\endgroup

\begin{abstract}
    In this paper we consider the linear elasticity eigenvalue problem with different boundary conditions and examine two a posteriori error estimators,
    that compute solutions to local residual problems using a partition of unity and localized residuals.
    These a posteriori error estimators are known for the source problem, and in this paper 
    we investigate their reliability and efficiency for the given eigenvalue problem with Dirichlet, Neumann and gliding boundary conditions.
    In numerical experiments we test the performance of the estimators on different domains, with different boundary conditions, and with different polynomial degrees. 
    The numerical experiments show efficiency indices closer to one compared to the standard residual-based a posteriori error estimator.
\end{abstract}

\section{Introduction}
    
    In this paper we examine a posteriori error estimators based on solving local problems in primal formulation based on a partition of unity approach for the following eigenvalue problem of linear elasticity.

    Let $\domainrtwo$ be a Lipschitz domain with polygonal boundary $\partial \Omega$. 
    We denote the Cauchy stress tensor as $\sigmau \coloneqq \C \, \epsu,$
    where $\C$ is the stiffness tensor, $\epsu \coloneqq \frac{1}{2} (\nabla \ufett + \nabla \ufett^T)$ is the linear strain tensor, and $\ufett$ is the displacement field.
    We partition $\partial \Omega$ into three parts: the closed Dirichlet part $\GammaD$, the closed gliding part $\Gammag$ and the open Neumann part $\GammaN \coloneqq \partial \Omega \setminus \big( \GammaD \cup \Gammag \big)$.
    Note that each of these parts should be connected and they should not overlap in more than a vertex.
    Furthermore, up to two of these boundary parts may be empty sets.
    
    Consider the eigenvalue problem: Find an eigenpair $(\kappa, \ufett) \in \R_{\geq 0} \times C^2(\Omega,\R^2) \cap C^1(\overline{\Omega},\R^2)$ such that
    \begin{alignat*}{2}
        - \divsigmau &= \kappa \, \ufett \quad &&\text{ in } \Omega, \\
        \sigmau \; \nfett &= \nullfett &&\text{ on } \GammaN, \\ 
        \ufett  &= \nullfett &&\text{ on } \GammaD,  \\ 
        \ufett \cdot \nfett = 0 \ \text{and} \ \tfett^T\; \sigmau \; \nfett &= 0 &&\text{ on } \Gammag,
    \end{alignat*}
    where $\nfett$ is the outer unit normal vector and $\tfett$ is the unit tangent vector. For simplicity we assume throughout this paper that the eigenvalue $\kappa$ is simple.

    A posteriori error estimators for the linear elasticity eigenvalue problem have previously been studied, e.g. in \cite{cg_2011}.
    They consider symmetric eigenvalue problems with Dirichlet boundary conditions in a general formulation and focus on first-order finite element methods.
    Two estimators are presented, where one focuses on edge contributions and the other is based on averaging techniques.
    In \cite{MR_A_priori_posteriori_error_estimates} an a posteriori error estimator is computed using virtual element methods.
    The mixed formulation was studied e.g.\ in  \cite{LRV_a_posteriori_analysis_for_mixed_FEM} with Dirichlet boundary conditions, where an a posteriori error estimator is proposed and its reliability and efficiency are proven. They also include a postprocessing.
    Bounds on the eigenvalues were computed, e.g.\ by \cite{LPZ_guaranteed_two_sided_bounds}, where they study two-sided bounds for elliptic systems, including results for linear elasticity eigenvalues. \\
    
    In this work we examine a posteriori error estimators based on solving local problems.
    These local problems are stated in displacement formulation and therefore their solutions may only be approximated with finite elements of higher polynomial degrees. 
    Our numerical examples illustrate that the mere increase of the polynomial degree by one is sufficient in all conducted numerical experiments. The use of the local displacement formulation allows for much simpler finite elements than a local mixed formulation, since mixed finite elements for linear elasticity problems are much more complicated. 
    
    The main idea for the estimators inspired by \cite{BR_err_est_for_AFEM_comp} is the use of a partition of unity to localize the residuals. This idea has first been derived for the Poisson problem in \cite{CF_fully_reliable_localized_err_control_in_the_FEM} and then been extended to the linear elasticity boundary value problem in \cite{error_estimator_pou}.  
    
    In \cite{error_estimator_pou}, the two estimators $\etalew$ and $\mulew$ are defined locally
        \begin{align*}
    	   \etazew &\coloneqq \sup_{\vfett \in \Hfett, \, \|\phiz^{\sfrac{1}{2}}\, D\vfett\|_z \neq 0} \ \frac{Res_z (\vfett)}{\|\phiz^{\sfrac{1}{2}}\, D\vfett\|_z} \\
    	    \muzew &\coloneqq \sup_{\vfett \in \Hfett, \ \|\phiz^{\sfrac{1}{2}} \, \C^{\sfrac{1}{2}} \, \varepsilon(\vfett)\|_z  \neq 0} \ \frac{Res_z (\vfett)}{\|\phiz^{\sfrac{1}{2}} \, \C^{\sfrac{1}{2}} \, \varepsilon(\vfett)\|_z },
        \end{align*}
        with $\Hfett$ being the solution space defined in Definition \ref{def: Hfett}, $Res_z$ being a local residual defined on a patch $\omega_z$, $\|\cdot\|_z$ being the energy norm and $\phiz$ the hat function on that patch, all specified in Section \ref{subsection: def residual}.
    	Summing up these local estimators over all nodes $\Ncal$ we obtain the global estimators
       \begin{equation*}
           \etalew^2 \coloneqq \sum_{z \in \Ncal} \etazew^2 \qquad \text{and} \qquad  \mulew^2 \coloneqq \sum_{z \in \Ncal} \muzew^2.
       \end{equation*}
    
    For the eigenpairs $(\kappa, \ufett)$ and $(\kappah, \ufetth)$ of the continuous and discrete eigenvalue problems we prove (global) reliability of the two estimators
        \begin{align*}
             |||\ufett - \ufetth|||  &\leq \sqrt{3} \, \CKorn \, \etalew + \frac{\kappa + \kappah}{2 \, \sqrt{\kappa}} \, \|\ufett - \ufetth\|_0, \\
             |||\ufett - \ufetth||| &\leq \sqrt{3} \, \mulew + \frac{\kappa + \kappah}{2 \, \sqrt{\kappa}} \, \|\ufett - \ufetth\|_0,
        \end{align*}
    in Theorems \ref{thm: reliability of etalew} and \ref{thm: reliability of mulew}, respectively, where $|||\cdot|||$ is the energy norm and $\CKorn$ is a constant depending solely on $\Omega$.
    
    Both estimators are also proven to be locally efficient
    \begin{align*}
        \etazew &\leq \Azeta \|\C^{\sfrac{1}{2}}\epse\|_z + \CPF \, h_z^2 \, \left( |\kappa - \kappah|   + \kappa \, \|\ufett - \ufetth\|_z \right), \\
        \muzew &\leq \Azmu \| \C^{\sfrac{1}{2}} \epse \|_z + C_z \, h_z^2 \, \left( |\kappa - \kappah|   + \kappa \, \|\ufett - \ufetth\|_z \right),
    \end{align*}
    for some local patch constants $\Azeta, \Azmu$ depending on the critical Lam\`e parameter $\lambda$, a constant $C_z$ not critical in $\lambda\to\infty$,  the mesh size $h_z$, and the Poincar\`e-Friedrich constant $\CPF$, which are specified in Theorems \ref{thm: efficiency of etalew} and \ref{thm: efficiency of mulew}, respectively.\\

    The remainder of the paper is organized as follows.
    In Section \ref{section: Weak formulation} we define the problem, derive a weak formulation and a discrete problem for all boundary conditions and mention some a priori analysis results.
    Afterwards, in Section \ref{section: A posteriori error analysis}, we start by introducing necessary notations and definitions for the estimators, including the residual and its local version together with the chosen partition of unity.
    We then define first $\etalew$ and then $\mulew$, show their reliability and efficiency and give a brief overview on how to compute the estimators in practice.
    Finally, the results of our numerical experiments are presented in Section \ref{section: experiments}.
    

    \section{Weak formulation}
    \label{section: Weak formulation}
    
    We start by defining a space $\Hfett$ that differs depending on the boundary condition we consider.
    \begin{definition}
        \label{def: Hfett}
        The solution space $\Hfett$ is given by
        \begin{equation*}
            \Hfett \coloneqq \left\{ \ufett \in H^1(\Omega, \, \R^2): \ \begin{array}{rl} \ufett  &\!\!\!\!= \nullfett \text{ on } \GammaD,  \\ \ufett \cdot \nfett & \!\!\!\!= 0 \text{ on } \Gammag \end{array} \right\}.
        \end{equation*}
    \end{definition}
    Using this space we derive a weak formulation.
    
    Integration by parts leads to
    
    \begin{equation}
        \label{eq: weak formulation mit Randintegralen}
        \intomega \sigmau : \epsv \; dx = \intomega \kappa \; \ufett \cdot \vfett \; dx + \intrandomega  (\sigmau \; \nfett) \cdot \vfett \; ds 
    \end{equation}
    for all $\vfett \in \Hfett$.
    
    Using the three boundary parts $\partial \Omega = \GammaN \cup \GammaD \cup \Gammag$ we examine their respective contributions to the boundary integral in \eqref{eq: weak formulation mit Randintegralen}.
    
    For Neumann boundary conditions we require $\vfett$ to be traciton-free, i.e.\ $\sigmau \, \nfett = \nullfett$ on $\GammaN$.
   
    The Dirichlet boundary condition is incorparated in $\Hfett$.
    Hence,
    \begin{equation*}
        \intrandomega (\sigmau \; \nfett) \cdot \vfett \; ds = \intrandomegag (\sigmau \; \nfett) \cdot \vfett \; ds .
    \end{equation*}
    
    We use the gliding boundary conditions $\vfett \cdot \nfett = 0$ and $\tfett^T\; \sigmau \; \nfett = 0$ on $\Gammag$ to model contact without friction.
   
    The first equation is called \emph{contact boundary condition} and the second one \emph{no friction boundary condition}, respectively.
    Decompositing $\vfett$ into a tangential and a normal component we obtain
    \begin{align*}
        &\quad \, \intrandomegag (\sigmau \; \nfett) \cdot \vfett \; ds = \intrandomegag (\sigmau \; \nfett) \cdot \big( (\tfett \cdot \vfett) \, \tfett + (\nfett \cdot \vfett) \, \nfett \big) \; ds \\
        &=  \intrandomegag (\tfett \cdot \vfett) \; (\sigmau \; \nfett) \cdot \tfett \; ds =  \intrandomegag (\tfett \cdot \vfett) \; \tfett^T \, \sigmau \; \nfett \; ds = 0.
    \end{align*}
    
    Overall, the integral over $\partial \Omega$ in \eqref{eq: weak formulation mit Randintegralen} vanishes on all three boundary parts.
    Next, we state the weak formulation.

    \begin{weak_formulation}
        \label{formulation: weak eigenv pb formul}
        Find $\ufett \in \Hfett$ and $\kappa \in \R_{\geq 0}$ such that $\|\ufett\|_0 = 1$ and
        \begin{equation}
            \label{eq: weak formulation a(u,v) = (ku,v)+rand}
            \arund(\ufett,\vfett) = (\kappa \,\ufett, \vfett)_{0}
        \end{equation}
        for all $\vfett \in \Hfett$, where the bilinear form $\arund: \Hfett \times \Hfett \to \R$ is defined as
        \begin{equation*}
            \arund(\ufett,\vfett) = \intomega \sigmau : \epsv \; dx,
        \end{equation*}
        $\sigmau \coloneqq \Cepsu$, and $(\cdot, \cdot)_{0}$ denotes the $\Lfett^2(\Omega,\R^2)$-inner product with the corresponding norm $\| \cdot \|_0$.
        
    \end{weak_formulation}
    
    For the discrete problem the space of piecewise polynomials is given by
        \begin{align*}
            \Pcal_k(\Tcal) &\coloneqq \left\{ \ufett \in L^2(\Omega,\R^2): \ \ufett_i|_T, \ i=1,2, \text{ is a polynomial} \right. \\
            &\quad\quad \left. \text{of total degree at most } k \text{ for all } T \in \Tcal \right\}.
        \end{align*}
        The discrete solution space is then defined as
        \begin{equation*}
            \Hfett_h \coloneqq \Hfett \cap \Pcal_k(\Tcal),
        \end{equation*}
        and the discrete problem follows from the continuous one in \eqref{eq: weak formulation a(u,v) = (ku,v)+rand}.
    \begin{discrete_problem}
        \label{formulation: discrete weak eigenv pb formul}
        Find $\ufetth \in \Hfett_h$ and $\kappah \in \R_{\geq 0}$ such that $\| \ufetth \|_0 =1$ and
        \begin{equation}
            \label{eq: diskr weak formulation a(u,v) = (ku,v)+rand}
            \arund(\ufetth,\vfetth) = (\kappah \ufetth, \vfetth)_{0}
        \end{equation}
        for all $\vfett_h \in \Hfett_h$, where
        \begin{equation*}
            \arund(\ufetth,\vfetth) = \intomega \sigmauh : \epsvh \; dx,
        \end{equation*}
        and $\sigmauh \coloneqq \Cepsuh$.
    \end{discrete_problem}
    
    It is well established that \eqref{eq: weak formulation a(u,v) = (ku,v)+rand} and \eqref{eq: diskr weak formulation a(u,v) = (ku,v)+rand} yield positive eigenvalues $\kappa,\kappah \geq 0$ which can be ordered increasingly.
    As for the Laplace eigenvalue problem in \cite{boffi}, we have a priori convergence rates given a regularity parameter $s > 1$ such that $\ufett \in H^{s}(\Omega)$
    \begin{align}
        \label{eq: a priori convergence rates}
        ||| \ufett - \ufetth ||| &= \mathcal{O}(h^{ \min (k, s-1)}), \notag \\ 
        \| \ufett - \ufetth \|_0 &= \mathcal{O}(h^{\min(1,s-1) + \min (k, s-1))}),\\
        |\kappa - \kappah| &= \mathcal{O}(h^{2 \min (k, s-1)}). \notag
    \end{align}
    In particular, the eigenvalues converge with twice the rate compared to the eigenfunctions.
    Due to the Rayleigh-Ritz principle we know that $\kappah \geq \kappa$.



\section{A posteriori error analysis}
\label{section: A posteriori error analysis}

    Throughout this paper we use the following notation.
    We assume that $\Tcal$ is a regular, quasi-uniform triangulation of our Lipschitz domain $\domainrtwo$, i.e.\ $\bigcup_{T \in \Tcal} T = \Omega \cup \partial \Omega$ with mesh size at most $h$.
    By $\Ecal$ we denote the set of all edges and by $\Ncal$ the set of all vertices of triangles in $\Tcal$.
    Adding the sub-indices $\Omega$, $N$, $D$ or $g$ describes inner, Neumann, Dirichlet or gliding edges and nodes, respectively.
    We use $T$ to describe triangles, $E$ for edges and $z$ for nodes.
    The mesh size of a specific triangle $T$ is denoted by $h_T$ and the mesh size of a specific edge is denoted by $h_E$, where due to the shape regularity of the triangulation we have that $h_T \approx h_E$.
    The abbreviations $x \lesssim y$ and $x \approx y$ stand for $x \leq c_1 \, y$ and $c_2 \, x \leq y \leq c_3 \, x $ respectively, with $c_1, c_2, c_3> 0$ some constants independent of $h$.


\subsection{Residual for the eigenvalue problem}
\label{subsection: def residual}
    
    Using the notation from above we now define the residual from which the error estimators are derived.

    Let $(\kappah, \ufetth)$ be the eigenpair to the Discrete Problem \ref{formulation: discrete weak eigenv pb formul}.
    Then the corresponding residual is defined as
    \begin{equation}
        Res_h:\Hfett \to \R,\quad
        \label{eq: def residual for eigenvalue problem}
        Res_h(\cdot) \coloneqq \kappah (\ufetth, \cdot)_0 - \arund(\ufetth, \cdot).
    \end{equation}
    Hence, $Res_h \in \Hfett^*$ is a linear functional with $Res_h(\vfetth) = 0$ for all $\vfetth \in \Hfett_h$.

    For some $\vfett \in \Hfett$ we rewrite the residual using integration by parts as
    \begin{equation}
        \label{eq: Resh in R and J}
        Res_h(\vfett) = \int_{\Omega} (\kappah \ufetth + \divsigmauh) \cdot \vfett \, dx - \sum_{E \in \mathcal{E}} \int_E ([\sigmauh] \; \nfett_E) \cdot \vfett \, ds
    \end{equation}
    where 
    \begin{align*}
	       [\sigmauh] \; \nfett_E
	       \coloneqq \begin{cases}
	           \nullfett &\ \text{for } E \in \Ecal_D, \\
	           \sigmauh \; \nfett_E \ &\ \text{for } E \in \Ecal_N, \\
	           \sigmauh \; \nfett_E &\ \text{for } E \in \Ecal_g, \\
	           (\sigmauh^+ |_{T_+}  - \ \sigmauh^- |_{T_-}) \; \nfett_E &\ \text{for } E=\partial T_+ \cap \partial T_- \in \Ecal_{\Omega},
	       \end{cases}
	\end{align*}
    denotes the normal jump of $\sigmauh = \C \, \epsuh$ for a fixed edge normal $\nfett_E$.
    
    \begin{remark}
        We see that for the jump of gliding edges $E \in \Ecal_g$, rewriting $\vfett \in \Hfett$ using the tangential and normal components and using the contact boundary condition $\vfett \cdot \nfett = 0$ yields
        \begin{align*}
            &\int_E  ([\sigmauh] \; \nfett_E) \cdot \vfett \, ds
            = \int_E (\sigmauh \; \nfett_E) \cdot \big( (\vfett \cdot \tfett_E) \, \tfett_E + (\vfett \cdot \nfett_E) \, \nfett_E \big) \, ds \\
            &= \int_E (\sigmauh \; \nfett_E) \cdot \big( (\vfett \cdot \tfett_E) \, \tfett_E \big) \, ds 
            = \int_E \tfett_E^T \, \sigmauh \, \nfett_E \, (\vfett \cdot \tfett_E) \, ds,
        \end{align*}
        which corresponds to the weakly incorporated no friction boundary condition $\tfett^T\; \sigmau \; \nfett = 0$.
    \end{remark}
    
    For the definition of the estimators we need two different norms of the residual.
    We define the \emph{dual norms} of the residual with respect to the energy norm $|||\cdot||| \coloneqq \bilform^{\sfrac{1}{2}}$ and the $\Hfett^1$-semi norm $\|D\cdot\|_0$ as
    \begin{equation}
        \label{eq: def dual norm res}
	   |||Res_h|||_{*} := \sup_{\substack{\vfett \in \Hfett \\ |||\vfett||| \neq 0}} \frac{Res_h(\vfett)}{|||\vfett|||}
	   \quad \text{and} \quad |||Res_h|||_{-1} := \sup_{\substack{\vfett \in \Hfett \\ \|D\vfett\|_0 \neq 0}} \frac{Res_h(\vfett)}{\|D\vfett\|_0},
    \end{equation}
    respectively.

    To connect these two norms, consider the following Korn's inequality from \cite[Theorem 11.2.12]{brenner_scott} for all $\vfett \in \Hfett$ with $|||\vfett||| \neq 0$
    \begin{equation*}
        \|D \vfett\|_0 \leq C_{\text{Kornelast}}\, |||\vfett|||,
    \end{equation*}
    where $C_{\text{Kornelast}}>0$ is a constant depending solely on $\Omega$.
    As the kernel of $||| \cdot |||$ coincides with the space of rigid body motions $\RBM (\Omega)$, the condition $||| \vfett ||| \neq 0$ implies $\|D \vfett\|_0 \neq 0$ and the given suprema \eqref{eq: def dual norm res} yield
    \begin{equation}
        \label{eq: resh dual norm korn elast leq resh -1 norm}
        |||Res_h|||_{*} \leq C_{\text{Kornelast}}\, |||Res_h|||_{-1}.
    \end{equation}


    To connect the dual norms to the error, we need the following equations.
    
    \begin{proposition}[{\cite[Lemma 6.3]{strang_fix}},{\cite[Lemma 3.1]{cg_2011}}]
        \label{lemma: gedicke L3.1 Teil 1}
        Let $(\kappa, \ufett)$ and $(\kappa_h, \ufett_h)$ be the eigenpairs of the continuous and the discrete eigenvalue problems \eqref{eq: weak formulation a(u,v) = (ku,v)+rand} and \eqref{eq: diskr weak formulation a(u,v) = (ku,v)+rand}, respectively.
        Then the energy norm of the error $\efett = \ufett - \ufett_h$ can be written as
        \begin{equation}
            \label{eq: CG Paper L3.1  eq 1}
            |||\efett|||^2 = \kappa \, \|\efett\|_0^2 + \kappa_h - \kappa.
        \end{equation}
    \end{proposition}
    
    \begin{proposition}[{\cite[Lemma 3.1, 3.2]{dpr}},{\cite[Lemma 3.1]{cg_2011}}]
        \label{lemma: gedicke L3.1 Teil 2}
        Using the same assumptions as in Proposition \ref{lemma: gedicke L3.1 Teil 1}, it follows that
        \begin{equation*}
            |||\efett|||^2 = \frac{(\kappa + \kappa_h)}{2} \, \|\efett\|_0^2 + Res_h(\efett).
        \end{equation*}
    \end{proposition}
    
    As a consequence, we have that
    \begin{equation}
        |||\efett|||^2 \geq \kappah - \kappa,
        \quad 
        \frac{1}{\sqrt{\kappa}} \, |||\efett||| \geq  \|\efett\|_0
        \label{eq: gedicke L3.1 Teil 2}
        \quad \text{and} \quad
        |||\efett||| \leq \frac{\kappa + \kappah}{2 \, \sqrt{\kappa}} \, \|\efett\|_0 + |||Res_h|||_*.
    \end{equation}
    
    These estimates are needed in the proofs of reliability for both estimators.


    To derive the error estimators we introduce a partition of unity.
    
    First we define the hat functions $\{\phiz\}_{z \in \Ncal}$ for some node $z \in \Ncal$ as
    \begin{equation*}
        \phiz: \Omega \to \R, \quad \phiz(x) \coloneqq \begin{cases}
        1 \quad \text{if } x = z \in \Ncal,\\
        0 \quad \text{if }x \in \Ncal \setminus \{z\},
        \end{cases} 
    \end{equation*}
    and in between the nodes we interpolate linearly over every element $T \in \Tcal$ such that $\phiz \in \Pcal_1$ is a finite element function of degree one with $L^{\infty}$-norm $\| \phiz \|_{\infty} = 1$.\\
    The patch $\omega_z$ is defined as an open set
    \begin{equation*}
        \omega_z \coloneqq \{ x \in \Omega \, | \, \phiz(x) \neq 0 \}
    \end{equation*}
    with diameter $h_z$.
    By definition, the closure of $\omega_z$ is
    \begin{equation*}
        \bar{\omega}_z = \bigcup_{T \in \Tcal(z)} T \quad \text{with } \Tcal(z) \coloneqq \{T \in \Tcal \, | \, z \in T\}.
    \end{equation*}
    The set of edges connected to $z$ is denoted by $\Ecal_z$.

    Next, we apply this partition of unity to the residual to localize it.
    Let $\vfett \in \Hfett$ and let $z \in \Ncal$ be a node with hat function $\phiz$ on the corresponding patch $\omega_z$.
    We define the local residual as
      \begin{align}
        \label{eq: Definition Res_z}
        Res_z(\vfett) &\coloneqq Res_h(\phiz \vfett) \\
        \notag
        &\overset{\eqref{eq: Resh in R and J}}{=}
           \int_{\Omega} \phiz \, (\kappah \ufetth + \divsigmauh) \cdot \vfett \, dx - \sum_{E \in \mathcal{E}} \int_E \phiz \, ([\sigmauh] \; \nfett_E) \cdot \vfett \, ds ,
      \end{align}
      where we inserted the definition of the global residual \eqref{eq: Resh in R and J}.

    We also need energy and $\Lfett^2$-norms on these local patches. 
    On a fixed patch $\omega_z$ we define the $\Lfett^2(\omega_z)$-inner product and norm as
    \begin{equation*}
        (\vfett,\wfett)_z \coloneqq \int_{\omega_z} \vfett \cdot \wfett \, dx    \quad \text{ and } \quad
        \|\vfett\|_z^2 \coloneqq (\vfett,\vfett)_z.
    \end{equation*}
    Similarly, we define the local bilinear form $\arundz: \Hfett \times \Hfett \to \R$ and its local energy norm as
    \begin{equation*}
        \arundz(\ufett,\vfett) \coloneqq  (\sigmau, \epsv)_z, \quad \text{ and } \quad  |||\vfett|||_z^2 \coloneqq \arundz(\vfett,\vfett), 
    \end{equation*}
    with $\ufett, \vfett \in \Hfett$.

    
\subsection{Error estimator $\etalew$}
\label{subsection: etalew}

    Next, we define the error estimators. 
    The first error estimator is based on $||| Res_h |||_{-1}$.
   
    \begin{definition}[first error estimator $\etalew$, {\cite[Definition 3.1]{error_estimator_pou}}]
        \label{def: etazew}
        The local error estimator is defined as
    	\begin{equation*}
    	   \etazew \coloneqq \sup_{\vfett \in \Hfett, \, \|\phiz^{\sfrac{1}{2}}\, D\vfett\|_z \neq 0} \ \frac{Res_z (\vfett)}{\|\phiz^{\sfrac{1}{2}}\, D\vfett\|_z}.
        \end{equation*}
    	Summing up these local estimators we obtain the global estimator
       \begin{equation*}
           \etalew^2 \coloneqq \sum_{z \in \mathcal{N}} \etazew^2.
       \end{equation*}
    \end{definition}
    
\subsubsection{Reliability of $\etalew$}
    
    We analyse the reliability of $\etalew$, i.e.\ whether the estimator approximates the error from above. 
    
    \begin{theorem}[Reliability of $\etalew$]
        \label{thm: reliability of etalew}
        Let $(\kappa,\ufett)$ and $(\kappa_h,\ufett_h)$ be the eigenpairs of the eigenvalue problems \eqref{eq: weak formulation a(u,v) = (ku,v)+rand} and \eqref{eq: diskr weak formulation a(u,v) = (ku,v)+rand}, respectively.
        Then
        \begin{equation}
            \label{eq: etalew resh norm abschaetzungen}
            |||Res_h|||_{-1} \leq \sqrt{3} \, \etalew, \qquad |||Res_h|||_* \leq \sqrt{3}\,  \CKorn \, \etalew,
        \end{equation}
        and for the error $\efett = \ufett - \ufett_h$
        \begin{equation}
            \label{eq: etalew reliability}
             |||\efett|||  \leq \sqrt{3} \, \CKorn \, \etalew + \frac{\kappa + \kappah}{2 \, \sqrt{\kappa}} \, \|\efett\|_0.
        \end{equation}
    \end{theorem}
    \begin{proof}
        Using \eqref{eq: resh dual norm korn elast leq resh -1 norm}, which states that $|||Res_h|||_{-1} \CKorn \geq |||Res|||_*$,
        the right equation for $\etalew$ in \eqref{eq: etalew resh norm abschaetzungen} follows from the left one. 
        Equation \eqref{eq: etalew reliability} follows using the bounds from the third equation in \eqref{eq: gedicke L3.1 Teil 2} together with the right equation in \eqref{eq: etalew resh norm abschaetzungen}, so we only need to show the left equation of \eqref{eq: etalew resh norm abschaetzungen}.\\
        Let $\vfett \in \Hfett, \ \|D \vfett\|_0 \neq 0$. 
        Consider the residual and split it into local residuals using the hat functions $\phiz$. 
        Insert the definition of $\etazew$ into the equation, then use a Cauchy-Schwarz inequality ($\CS$), a Hölder inequality (H.) and $\|\phiz\|_{\infty} =1$ to obtain
        \begin{align*}
            Res_h(\vfett) &\overset{\eqref{eq: Resh in R and J}}{=} \int_{\Omega} (\kappah \ufetth + \divsigmauh) \cdot \vfett \, dx - \sum_{E \in \mathcal{E}} \int_E ([\sigmauh] \; \nfett_E) \cdot \vfett \, ds \\
            &= \sum_{z \in \Ncal} \Big( \int_{\Omega} \phiz \, (\kappah \ufetth + \divsigmauh) \cdot \vfett \, dx - \sum_{E \in \mathcal{E}} \int_E \phiz \, ([\sigmauh] \; \nfett_E) \cdot \vfett \, ds \Big) \\
            &= \sum_{z \in \Ncal} Res_z(\vfett)
            \leq \sum_{z \in \Ncal} \etazew \| \phiz^{\sfrac{1}{2}} D\vfett \|_z \\
            &\overset{\CS}{\leq} \Big( \sum_{z \in \Ncal} \etazew^2 \Big)^{\sfrac{1}{2}} \, \Big( \sum_{z \in \Ncal} \|\phiz^{\sfrac{1}{2}} D\vfett\|^2_z \Big)^{\sfrac{1}{2}} 
            \overset{\text{H.}}{\leq}  \big(\etalew^2\big)^{\sfrac{1}{2}}  \, \Big( \sum_{z \in \Ncal} \|\phiz\|_{\infty} \, \|D\vfett\|^2_z \Big)^{\sfrac{1}{2}}\\
            &\leq \sqrt{3} \, \etalew \, \|D\vfett\|_0,
        \end{align*}
        where  the number 3 stems from the overlap of the vertex patches.
        In conclusion, we have for any $\vfett \in \Hfett$, $\|D\vfett \|_0 \neq 0$,
        \begin{equation*}
            Res_h(\vfett) \leq \sqrt{3} \, \etalew \, \|D\vfett\|_0,
        \end{equation*}
        and therefore the same equation holds for the supremum, which implies \eqref{eq: etalew resh norm abschaetzungen}.
    \end{proof}
    
\subsubsection{Efficiency of $\etalew$}
\label{subsubsection: eff of etalew}
    
    To prove the efficiency of $\etalew$ we use a weighted Poincaré-Friedrichs inequality from \cite{error_estimator_pou} that was proven e.g.\ in \cite{CF_fully_reliable_localized_err_control_in_the_FEM, MNS_Local_pb_stars}. For this inequality we need to define the following spaces.
    
    \begin{definition}
        \label{def: Vz fuer etazew}
        Fix a node $z \in \Ncal$. Then depending on the location of $z$ and the boundary condition we define
        \begin{enumerate}
            \item[a)] for $z \in \Ncal_{\Omega}$ or $z \in \Ncal_N$ 
            \begin{equation*}
                \Vfett_z \coloneqq \left\{ \vfett \in H^1(\omega_z, \R^2) :\ \int_{\omega_z} \vfett \, dx = \nullfett \right\},
            \end{equation*}
            
            \item[b)] for $z \in \Ncal_D$ or $z \in \Ncal_g$ with $z$ being on a corner
            \begin{equation*}
                \Vfett_z \coloneqq \left\{ \vfett \in H^1(\omega_z, \R^2) :\ \vfett = \nullfett \text{ on } \Ecal_z \cap \Ecal_D \text{ and } \vfett \cdot \nfett = 0 \text{ on } \Ecal_z \cap \Ecal_g \right\},
            \end{equation*}
            \item[c)] \label{eq: def Vz fuer etazew gliding}
            for $z \in \Ncal_g$ with $z$ not being on a corner
            \begin{equation*}
                \Vfett_z \coloneqq \left\{ \vfett \in H^1(\omega_z, \R^2) :\ \vfett \cdot \nfett = 0 \text{ on } \Ecal_z \cap \Ecal_g \text{ and } \int_{\Ecal_z \cap \Ecal_g} \vfett \cdot \tfett \, dx = 0 \right\}
            \end{equation*}
            with $\tfett$ the unit tangent vector.
        \end{enumerate}
    \end{definition}
    \begin{remark}
        Using this space $\Vfett_z$ we rewrite $\etazew$ as
        \begin{equation}
            \etazew = \sup_{\vfett \in \Vfett_z \setminus \{\nullfett\}} \frac{Res_z (\vfett)}{\|\phiz^{\sfrac{1}{2}}\, D\vfett\|_z}.
            \label{eq: comp of etazew defi etazew neu}
        \end{equation}
        This coincides with Definition \ref{def: etazew} as $\|\phiz^{\sfrac{1}{2}} \, D\vfett\|_z \neq 0$ is equivalent to $\vfett$ not being constant on $\omega_z$, e.g.\ in case a) $\int_{\omega_z} \vfett \, dx = 0 , \ \vfett \neq \nullfett$ and similar in cases b) and c).
    \end{remark}
    
    \begin{proposition}[Weighted Poincaré-Friedrichs inequality, {\cite[Theorem 3.2]{error_estimator_pou}}]
        Let $\vfett \in \Vfett_z$. Then
        \begin{equation}
            \label{eq: weighted Poincare-Friedrichs}
            \|\vfett\|_z \leq \CPF\, h_z \, \|\phiz^{\sfrac{1}{2}}\, D\vfett\|_z
        \end{equation}
        with $\CPF$ being a constant that depends on the shape and the boundary conditions of the patch $\omega_z$, but not on its size.
    \end{proposition}
    \begin{proof}
        See \cite[Lemmata 5.1, 5.3]{CF_fully_reliable_localized_err_control_in_the_FEM}, \cite[Proposition 2.4]{MNS_Local_pb_stars}.
    \end{proof}
    
    Using this inequality we prove the efficiency of $\etalew$.
    \begin{theorem}[Efficiency of $\etalew$]
        \label{thm: efficiency of etalew}
        For any node $z \in \Ncal$ using the local constant
        \begin{equation*}
            \Azeta \coloneqq \sup_{\vfett \in \Vfett_z \setminus \{\nullfett\}} \|\C^{\sfrac{1}{2}} \varepsilon(\phiz \vfett)\|_z \ \big/ \ \|\phiz^{\sfrac{1}{2}} D\vfett\|_z \ < \infty
        \end{equation*}
        the local error estimator $\etazew$ is bounded by
        \begin{equation}
            \label{eq: lokale efficiency etazew}
            \etazew \leq \Azeta \|\C^{\sfrac{1}{2}}\epse\|_z + \CPF \, h_z^2 \, \big( |\kappa  - \kappah| +  \kappa \, \|\efett\|_z \big).
        \end{equation}
        Defining $C_A \coloneqq \max_{z \in \Ncal} \Azeta$, we conclude global efficiency of $\etalew$ in the sense that
        \begin{equation*}
            \etalew \leq \sqrt{6} \: C_A \, |||\efett||| + \sqrt{6} \, \CPF \, h^2 \, \big( |\kappa - \kappah| + \kappa \, \|\efett\|_0 \big),
        \end{equation*}
        with $h \coloneqq \max_{z \in \Ncal} \, h_z$.
    \end{theorem}
    
    \begin{proof}
        The proof of the estimate for $\Azeta < \infty$ can be found in \cite[Theorem 3.3]{error_estimator_pou}. \\
        We show local efficiency first.
        Let $\vfett \in \Vfett_z \setminus \{\nullfett\}$ be arbitrary. 
        Then the definitions of the local residual \eqref{eq: def residual for eigenvalue problem} and $\Azeta$ together with the weak formulation \eqref{eq: weak formulation a(u,v) = (ku,v)+rand} and a Cauchy-Schwarz inequality ($\CS$) yield
        \begin{align}
            \notag
            Res_z(\vfett) &= Res_h(\phiz \vfett) \\
            \notag
            &\overset{\eqref{eq: def residual for eigenvalue problem}}{=} \int_{\omega_z} \kappah\,  \ufetth \cdot \phiz \vfett \, dx - \int_{\omega_z} \C \varepsilon(\ufett_h): \varepsilon(\phiz \vfett) dx \\
            \label{eq: proof eff etazew - Resz als Ceps(e) + khuh-ku}
            &\overset{\eqref{eq: weak formulation a(u,v) = (ku,v)+rand}}{=} \int_{\omega_z} \C \varepsilon(\ufett - \ufett_h): \varepsilon(\phiz \vfett) dx+ \int_{\omega_z}   (\kappah\ufetth - \kappa \, \ufett) \cdot \phiz \vfett \, dx\\
            \notag
            &\overset{\CS}{\leq} \|\C^{\sfrac{1}{2}} \epse\|_z \|\C^{\sfrac{1}{2}} \varepsilon(\phiz \vfett)\|_z + \|\kappah\ufetth - \kappa \, \ufett\|_z \, \|\phiz \vfett\|_z\\
            \label{eq: proof efficiency etalew, Resh als summe von a() und restterm}
            & \overset{\text{Def.\ } \Azeta}{\leq} \|\C^{\sfrac{1}{2}} \epse\|_z\, \Azeta \, \|\phiz^{\sfrac{1}{2}} D \vfett\|_z +   \|\kappa\,\ufett - \kappah \ufetth\|_z \, \|\phiz \vfett\|_z
        \end{align}
        To estimate $\|\phiz \vfett\|_z$ we use a Cauchy-Schwarz inequality and the weighted Poincaré-Friedrichs inequality \eqref{eq: weighted Poincare-Friedrichs}
        \begin{equation*}
            \|\phiz \vfett\|_z \leq \|\phiz\|_z \, \|\vfett\|_0 
            \leq h_z \|\vfett\|_0 
            \overset{\eqref{eq: weighted Poincare-Friedrichs}}{\leq} \CPF \, h_z^2 \, \|\phiz^{\sfrac{1}{2}} D\vfett\|_z.
        \end{equation*}
        Thus,
        \begin{equation*}
            Res_z(\vfett) \leq \|\C^{\sfrac{1}{2}} \epse\|_z\, \Azeta \, \|\phiz^{\sfrac{1}{2}} D \vfett\|_z + \CPF \,  h_z^2 \, \|\kappa\,\ufett - \kappah \ufetth\|_z \, \|\phiz^{\sfrac{1}{2}} D\vfett\|_z.
        \end{equation*}
        Since this holds for any $\vfett \in \Vfett_z \setminus \{ \nullfett \}$, we estimate $\etazew$ by dividing both sides by $\|\phiz^{\sfrac{1}{2}} D \vfett\|_z$ (which is not zero by definition of $\Vfett_z$) and obtain
        \begin{equation}
            \label{eq: efficiency etazew mit zusatzterm}
            \etazew \leq \Azeta \|\C^{\sfrac{1}{2}}\epse\|_z + \CPF \, h_z^2 \,  \|\kappa\,\ufett - \kappah  \ufetth\|_z.
        \end{equation}
        We use $\|\ufetth\|_z \leq \|\ufetth\|_0 = 1$ and $\kappa \geq 0$ to obtain
        \begin{align}
            \notag
            \| \kappa\, \ufett - \kappah \ufetth\|_z 
            &=  \|\kappa\, \ufett - \kappa \, \ufetth + \kappa \, \ufetth - \kappah \ufetth\|_z  \\
            \notag
            &\leq  \|\kappa\, \ufetth - \kappah \ufetth\|_z + \| \kappa \, \ufett - \kappa \, \ufetth\|_z  \\
            \label{eq: ku - kh uh Abschaetzung}
            &\leq |\kappa  - \kappah| +  \kappa \, \|\efett\|_z.
        \end{align}
        Together with \eqref{eq: efficiency etazew mit zusatzterm} this yields the local efficiency. \\
        For the global version of the estimator, note that due to a Young's inequality 
        \begin{align*}
            \etalew^2 &= \sum_{z \in \Ncal} \etazew^2 
            \leq \sum_{z \in \Ncal} \Big(  \Azeta \|\C^{\sfrac{1}{2}}\epse\|_z + \CPF \,  h_z^2 \, \|\kappa\,\ufett - \kappah \ufetth\|_z  \Big)^2\\
            &\leq \sum_{z \in \Ncal} 2\, C_A^2 \|\C^{\sfrac{1}{2}}\epse\|^2_z  + \sum_{z \in \Ncal} 2 \, \CPF^2 \,  h_z^4 \, \|\kappa\,\ufett - \kappah\ufetth\|_z^2.
        \end{align*}
        
        Estimating the overlap by $3$ and using $h$ as upper bound for any $h_z$ yields
        \begin{align*}
            \etalew &\leq \Big( \, 6\,  C_A^2 \, \| \C^{\sfrac{1}{2}}\epse\|^2_0  + 6 \, \CPF^2 \, h^4 \,  \|\kappa\,\ufett - \kappah \ufetth\|_0^2 \Big)^{\sfrac{1}{2}} \\
            &\leq \sqrt{6} \, C_A \, |||\efett||| + \sqrt{6}\, \CPF \, h^2 \,  \|\kappa\,\ufett - \kappah \ufetth\|_0 \\
            &\overset{\eqref{eq: ku - kh uh Abschaetzung} }{\leq} \sqrt{6} \, C_A \, |||\efett||| + \sqrt{6}\, \CPF \, h^2 \, \big( |\kappa - \kappah| + \kappa \, \|\efett\|_0 \big),
        \end{align*}
        which shows the efficiency of $\etalew$.
    \end{proof}
    
    In addition, we formulate the reliability and efficiency results with respect to the eigenvalue error.
    
    \begin{corollary}[Reliability and efficiency regarding the eigenvalue error]
        \label{cor: reliability and efficiency of etalew wrt kappa}
        Under the same assumptions as in Theorems \ref{thm: reliability of etalew} and \ref{thm: efficiency of etalew} it holds that $\etalew$ is reliable in the sense that
        \begin{equation*}
            |\kappa - \kappah| \leq 6 \, \CKorn^2 \, \etalew^2 + \frac{\kappa^2 + \kappah^2}{\kappa} \, \|\efett\|_0^2 
        \end{equation*}
        and efficient in the sense that
        \begin{equation*}
            \etalew^2 \leq 12 \, C_A^2 \, |\kappa - \kappah| + 24 \, \CPF^2 \, h^4 \, |\kappa-\kappah|^2 +  12 \, \|\efett\|_0^2 \, \left( C_A^2 \, \kappa  +  2 \, \CPF^2 \, h^4 \, \kappa^2 \right).
        \end{equation*}
    \end{corollary}
    \begin{proof}
        Reliability follows from the first inequality in \eqref{eq: gedicke L3.1 Teil 2} and the reliability of $\etalew$ in Theorem \ref{thm: reliability of etalew}.
        Efficiency follows from Proposition \ref{lemma: gedicke L3.1 Teil 1} and the efficiency of $\etalew$ in Theorem \ref{thm: efficiency of etalew}.
    \end{proof}
    
    \begin{remark}
        Note that all other terms than the eigenvalue error term $12 \, C_A^2 \, |\kappa - \kappah|$ are of higher order by \eqref{eq: a priori convergence rates}.
    \end{remark}


\subsection{Error estimator $\mulew$}
\label{subsection: mulew}
    
    In the following section we define and examine the error estimator $\mulew$.
    Unlike $\etalew$, this estimator's reliabilty does not rely on a global Korn constant.
    
    \begin{definition}[second error estimator $\mulew$, {\cite[Definition 4.1]{error_estimator_pou}}]
        \label{def: muzew}
        Similar to $\etalew$, the estimator $\mulew$ is defined using local components $\muzew$ on patches $\omega_z$ based on $|||Res_h|||_{*}$
    	  \begin{equation*}
    	       \muzew \coloneqq \sup_{\vfett \in \Hfett, \ \|\phiz^{\sfrac{1}{2}} \, \C^{\sfrac{1}{2}} \, \varepsilon(\vfett)\|_z  \neq 0} \ \frac{Res_z (\vfett)}{\|\phiz^{\sfrac{1}{2}} \, \C^{\sfrac{1}{2}} \, \varepsilon(\vfett)\|_z }.
    	   \end{equation*}
    	   Summing up these local error estimators, we obtain
    	  \begin{equation*}
    	        \mulew^2 \coloneqq \sum_{z \in \Ncal} \muzew^2.
    	   \end{equation*}
    \end{definition}
    
    \subsubsection{Reliability of $\mulew$}
    As for $\etalew$, we also show the reliability of $\mulew$.
    
    \begin{theorem}[Reliability of $\mulew$]
        \label{thm: reliability of mulew}
        The estimator $\mulew$ bounds the error from above
        \begin{equation}
            \label{eq: reliability mulew}
            |||\efett||| \leq \sqrt{3} \, \mulew + \frac{\kappa + \kappah}{2 \, \sqrt{\kappa}} \, \|\efett\|_0.
        \end{equation}
    \end{theorem}
    \begin{proof}
        We start by estimating $Res_h (\efett)$.
        To do so, we localize the residual, insert a local norm, and estimate it by the supremum to obtain $\muzew$, which leads to $\mulew$ using a Cauchy-Schwarz inequality ($\CS$) and the properties of the partition of unity
        \begin{align*}
             Res_h(\efett) &=  \sum_{z \in \Ncal} Res_z(\efett) \\
             &= \sum_{z \in \Ncal} \: \frac{Res_z (\efett)}{\|\varphi_z^{\sfrac{1}{2}} \C^{\sfrac{1}{2}}  \epse\|_z} \|\phiz^{\sfrac{1}{2}} \C^{\sfrac{1}{2}} \epse\|_z \\
             &\leq \sum_{z \in \Ncal} \: \sup_{\vfett \in \Hfett, \|\varphi_z^{\sfrac{1}{2}} \C^{\sfrac{1}{2}}  \varepsilon(\vfett)\|_z \neq 0} \, \frac{Res_z (\vfett)}{\|\varphi_z^{\sfrac{1}{2}} \C^{\sfrac{1}{2}}  \varepsilon(\vfett)\|_z} \, \|\phiz^{\sfrac{1}{2}} \C^{\sfrac{1}{2}} \epse\|_z \\
             &= \sum_{z \in \Ncal} \muzew \, \|\phiz^{\sfrac{1}{2}} \C^{\sfrac{1}{2}} \epse\|_z \\
             &\overset{\CS}{\leq} \left( \sum_{z 
             \in \Ncal} \muzew^2 \right)^{\sfrac{1}{2}} \, \left( \sum_{z 
             \in \Ncal} \|\phiz^{\sfrac{1}{2}} \C^{\sfrac{1}{2}} \epse\|_z^2 \right)^{\sfrac{1}{2}} \\
             &\overset{\text{Def. } \phiz}{\leq} \sqrt{3} \, \mulew \; \|\C^{\sfrac{1}{2}} \epse\|_0 = \sqrt{3} \, \mulew \, |||\efett|||.
        \end{align*}
        Next, we use Proposition \ref{lemma: gedicke L3.1 Teil 2} and \eqref{eq: gedicke L3.1 Teil 2} to estimate $||| \efett |||$, where we insert our estimate of $Res_h(\efett)$
        \begin{align*}
            ||| \efett |||^2 &=  \frac{(\kappa + \kappa_h)}{2} \, \|\efett\|_0^2 + Res_h(\efett) \\
            &\overset{\eqref{eq: gedicke L3.1 Teil 2}}{\leq}  \frac{(\kappa + \kappa_h)}{2 \, \sqrt{\kappa}} \, \|\efett\|_0 \, ||| \efett ||| + Res_h(\efett) \\
            &\leq \frac{(\kappa + \kappa_h)}{2 \, \sqrt{\kappa}} \, \|\efett\|_0 \, ||| \efett ||| + 3 \, \mulew \, ||| \efett |||.
        \end{align*}
        Dividing both sides by $|||\efett|||$ yields the desired result.
    \end{proof}
    
\subsubsection{Efficiency of $\mulew$}
\label{subsubsection: eff of mulew}
    
    As for $\etalew$, we require a local efficiency estimate for $\mulew$. Again, we define suitable local spaces.
    
    \begin{definition}
        \label{def: Wz fuer muzew}
        Fix a node $z \in \Ncal$. Then depending on the location of $\omega_z$ and the boundary condition define
        \begin{enumerate}
            \item[a)] for $z \in \Ncal_{\Omega}$ or $z \in \Ncal_N$ 
            \begin{equation*}
                \Wfett_z \coloneqq \left\{ \vfett \in H^1(\omega_z, \R^2): \ \int_{\omega_z} \vfett \, dx = \nullfett, \ \int_{\omega_z} \rot \vfett \, dx = \nullfett \right\},
            \end{equation*}
            \item[b)] for $z \in \Ncal_D$ or $z \in \Ncal_g$ with $z$ being on a corner
            \begin{equation*}
                \Wfett_z \coloneqq \left\{ \vfett \in H^1(\omega_z, \R^2): \ \vfett = \nullfett \text{ on } \Ecal_z \cap \Ecal_D \text{ and } \vfett \cdot \nfett = 0 \text{ on } \Ecal_z \cap \Ecal_g \right\}
            \end{equation*}
            \item[c)] \label{eq: def Wz fuer muzew gliding}
            for $z \in \Ncal_g$ with $z$ not being on a corner
            \begin{equation*}
                \Wfett_z \coloneqq \left\{ \vfett \in H^1(\omega_z, \R^2): \ \vfett \cdot \nfett = 0 \text{ on } \Ecal_z \cap \Ecal_g  \text{ and } \int_{\Ecal_z \cap \Ecal_g} \vfett \cdot \tfett \, dx = 0 \right\}
            \end{equation*}
            with $\tfett$ the unit tangent vector.
        \end{enumerate}
    \end{definition}

    \begin{remark}
        With this space $\Wfett_z$ we rewrite $\muzew$ as
        \begin{equation}
            \muzew = \sup_{\vfett \in \Wfett_z \setminus \{\nullfett\}} \, \frac{Res_z(\vfett)}{\|\phiz^{\sfrac{1}{2}} \, \C^{\sfrac{1}{2}} \, \epsv \|_z }.
            \label{eq: comp of muzew defi muzew neu}
        \end{equation}
        Note that $\|\phiz^{\sfrac{1}{2}} \, \C^{\sfrac{1}{2}} \, \epsv\|_z \neq 0$ is implied by $\vfett \in \Wfett_z \setminus \{\nullfett\}$.
        The definitions in b) and c) coincide with the definitions b) and c) of $\Vfett_z$ in Definition \ref{def: Vz fuer etazew}.
    \end{remark}
    
    Using this definition we need the following Korn's inequality to proof the efficiency of $\mulew$.
    
    \begin{proposition}[Weighted Korn's inequality, {\cite[Chapter 4.3]{error_estimator_pou}}]
        For $\vfett \in \Wfett_z$ there exists a constant $\CWKorn>0$ depending on the shape but not the size of the patch $\omega_z$, such that
        \begin{equation}
            \label{eq: weighted Korn's inequality}
            \|\phiz^{\sfrac{1}{2}} \, D\vfett\|_z \leq \CWKorn \, \|\phiz^{\sfrac{1}{2}} \, \varepsilon(\vfett)\|_z.
        \end{equation}
    \end{proposition}
    
    \begin{theorem}[Efficiency of $\mulew$]
        \label{thm: efficiency of mulew}
        Let $z \in \Ncal$ be a node. Using the constant
        \begin{equation*}
            \Azmu \coloneqq \sup_{\vfett \in \Wfett_z \setminus \{\nullfett\}} \, \|\C^{\sfrac{1}{2}} \varepsilon(\phiz \vfett) \|_z \, \big/ \, \|\C^{\sfrac{1}{2}} \phiz^{\sfrac{1}{2}} \varepsilon(\vfett)\|_z \ < \infty
        \end{equation*}
       we bound $\muzew$ by
        \begin{equation}
            \label{eq: efficiency muzew}
            \muzew \leq \Azmu \| \C^{\sfrac{1}{2}} \epse \|_z  + C_z \, h_z^2 \, \left( |\kappa - \kappah|   + \kappa \, \|\efett\|_z \right),
        \end{equation}
        where $C_z \coloneqq \CPF \, \CWKorn \, C'$, with $C'$ being the constant that bounds $1/|\C_{min}| \coloneqq 1 / \min_{i,j,k,l} |\C_{ijkl}|$.
        Since this holds for every node $z \in \Ncal$, we obtain a global bound
        \begin{equation*}
            \label{eq: efficiency mulew}
            \mulew \leq \sqrt{6} \, C_A \, |||\efett||| + \sqrt{6} \, C_L \, h^2 \, \left( |\kappa - \kappah| + \kappa \, \|\efett\|_0 \right)
        \end{equation*}
        with $C_A \coloneqq \max_{z \in \Ncal} \, \Azmu$, $C_L \coloneqq \max_{z \in \Ncal} \, C_z$ and $h \coloneqq \max_{z \in \Ncal} \, h_z$. 
        This estimate shows that $\mulew$ is efficient, as the second term is of higher order.
    \end{theorem}
    \begin{proof}
        For the proof of the estimate of $\Azmu < \infty$ we refer to \cite[Theorem 4.6]{error_estimator_pou}.
        \\ 
        For any $\vfett \in W_z$, using similar computations as in \eqref{eq: proof eff etazew - Resz als Ceps(e) + khuh-ku}
        leads to
        \begin{align}
            \notag
            Res_z(\vfett)
            &\overset{\eqref{eq: proof eff etazew - Resz als Ceps(e) + khuh-ku}}{=} \int_{\omega_z} \C\, \epse:\varepsilon(\phiz \vfett) \, dx  + \int_{\omega_z}  (\kappah \ufetth - \kappa \, \ufett) \cdot \phiz \vfett \, dx\\
            \label{eq: Proof Eff of mulew - Resz Abschaetzung}
            &\overset{\eqref{eq: proof of mulew efficiency hilfsgleichung}}{\leq}  \| \C^{\sfrac{1}{2}} \epse \|_z \, \Azmu \, \| \C^{\sfrac{1}{2}} \phiz^{\sfrac{1}{2}} \varepsilon(\vfett)\|_z + \|\kappa\, \ufett - \kappah \ufetth\|_z \, \|\phiz \vfett\|_z,
        \end{align}
        where we used the following estimate in the last step
        \begin{align}
            \notag
            &\ \int_{\omega_z} \C\, \epse:\varepsilon(\phiz \vfett) \, dx 
            \overset{\CS}{\leq} \| \C^{\sfrac{1}{2}} \epse \|_z \, \| \C^{\sfrac{1}{2}} \varepsilon(\phiz \vfett)\|_z \\
            &\overset{\text{Def.\ } \Azmu}{\leq}  \| \C^{\sfrac{1}{2}} \epse \|_z \, \Azmu \, \| \C^{\sfrac{1}{2}} \phiz^{\sfrac{1}{2}} \varepsilon(\vfett)\|_z .
            \label{eq: proof of mulew efficiency hilfsgleichung}
        \end{align}
        We estimate $\|\phiz \vfett\|_z$ using the weighted Poincaré-Friedrichs inequality \eqref{eq: weighted Poincare-Friedrichs}, the weighted Korn's inequality \eqref{eq: weighted Korn's inequality} and the boundedness of $\C^{-1}$ to obtain
        \begin{align}
            \notag
            \|\phiz \, \vfett\|_z &\leq \|\phiz\|_z \, \|\vfett\|_z 
            \leq h_z \|\vfett\|_z 
            \overset{\eqref{eq: weighted Poincare-Friedrichs}}{\leq} \CPF \, h_z^2 \, \|\phiz^{\sfrac{1}{2}} \, D\vfett\|_z \\
            \label{eq: Proof Eff of mulew - z Norm aufteilen}
            &\overset{\eqref{eq: weighted Korn's inequality}}{\leq} \CPF \, \CWKorn \, h_z^2 \, \|\phiz^{\sfrac{1}{2}} \, \varepsilon(\vfett)\|_z 
            \leq C_z \, h_z^2 \, \|\phiz^{\sfrac{1}{2}} \, \C^{\sfrac{1}{2}} \varepsilon(\vfett)\|_z.
        \end{align}
        Applying \eqref{eq: Proof Eff of mulew - Resz Abschaetzung} and \eqref{eq: Proof Eff of mulew - z Norm aufteilen} yields
        \begin{equation*}
            Res_z(\vfett) \leq \Azmu \| \C^{\sfrac{1}{2}} \epse \|_z \, \| \C^{\sfrac{1}{2}} \phiz^{\sfrac{1}{2}} \varepsilon(\vfett)\|_z +C_z \, h_z^2 \, \|\kappa\,\ufett - \kappah \ufetth\| \,  \| \phiz^{\sfrac{1}{2}}  \C^{\sfrac{1}{2}}  \varepsilon(\vfett)\|_z.
        \end{equation*}
        We divide by $\|\phiz^{\sfrac{1}{2}} \,\C^{\sfrac{1}{2}}\, \varepsilon(\vfett)\|_z$ on both sides. This estimate holds for any $\vfett \in W_z$, it also holds for the supremum, thus the definition of $\muzew$ yields
        \begin{equation*}
             \muzew \leq \Azmu \| \C^{\sfrac{1}{2}} \epse \|_z  + C_z \, h_z^2 \, \| \kappa\, \ufett - \kappah \ufetth\|_z.
        \end{equation*}
        Combined with \eqref{eq: ku - kh uh Abschaetzung}, this estimate yields \eqref{eq: efficiency muzew}.
        To obtain a global version, we sum up over all nodes and use \eqref{eq: efficiency muzew}) and the definition of $C_A$
        \begin{align*}
            \mulew^2 &= \sum_{z \in \Ncal} \muzew^2 \overset{\eqref{eq: efficiency muzew}}{\leq} \sum_{z \in \Ncal} \left( \Azmu \| \C^{\sfrac{1}{2}} \epse \|_z  + C_z \, h_z^2 \, \| \kappa\, \ufett - \kappah \ufetth\|_z \right)^2 \\
            &\overset{\text{Def. } C_A}{\leq}\sum_{z \in \Ncal} 2 \, C_A^2 \| \C^{\sfrac{1}{2}} \epse \|^2_z \, + 2 \, C_z^2 \, h_z^4 \, \| \kappa\, \ufett - \kappah \ufetth\|_z^2 .
        \end{align*}
        Using the overlap by 3, the definitions of $C_L$ and $h$, this yields
        \begin{align*}
            \mulew &\leq \left( 6 \, C_A^2 \| \C^{\sfrac{1}{2}} \epse \|^2_0 + 6 \, C_L^2 \, h^4 \, \| \kappa\, \ufett - \kappah \ufetth\|_0^2 \right)^{\sfrac{1}{2}} \\
            &\overset{\eqref{eq: ku - kh uh Abschaetzung}}{\leq} \sqrt{6}\, C_A \, |||\efett|||  + \sqrt{6} \, C_L \, h^2 \, \left( |\kappa - \kappah| + \kappa \, \|\efett\|_0 \right),
        \end{align*}
        which proves the efficiency of $\mulew$.
    \end{proof}
    
    Next we state reliability and efficiency estimates with respect to the eigenvalue error.
    
    \begin{corollary}[Reliability and efficiency regarding the eigenvalue error]
        Under the same assumptions as in Theorems \ref{thm: reliability of mulew} and \ref{thm: efficiency of mulew}, there holds reliability of $\mulew$ as
        \begin{equation*}
            |\kappa - \kappah| \leq 24 \, \mulew^2 + \frac{\kappa^2 + \kappah^2}{\kappa} \, \|\efett\|_0^2 
        \end{equation*}
        and efficiency of $\mulew$ as
        \begin{equation*}
            \mulew^2 \leq 12 \, C_A^2 \, |\kappa - \kappah| + 24 \, C_L^2 \, h^4 \, |\kappa-\kappah|^2 + 12 \, \|\efett\|_0^2 \, \left( C_A^2 \, \kappa  +  2 \, C_L^2\, h^4  \kappa^2 \right).
        \end{equation*}
    \end{corollary}
    \begin{proof}
        Reliability follows from the first inequality in \eqref{eq: gedicke L3.1 Teil 2} and the reliability of $\mulew$ in Theorem \ref{thm: reliability of mulew}.
        Efficiency follows from the efficiency of $\mulew$ in Theorem \ref{thm: efficiency of mulew} and Proposition \ref{lemma: gedicke L3.1 Teil 1}.
    \end{proof}
    
    \begin{remark}
        Note that for the efficiency all other terms than the eigenvalue error term $12 \, C_A^2 \, |\kappa - \kappah|$ are of higher order by \eqref{eq: a priori convergence rates}.
    \end{remark}


\subsection{Computation of $\etazew$ and $\muzew$}
\label{subsection: comp of etalew, mulew}

    The computation of $\etazew$ and $\muzew$ requires the approximation of suprema. 
    To compute a numerical approximation we reformulate these suprema into solutions of discrete local problems, similar to \cite[Chapter 5.1]{error_estimator_pou}.
    We denote the approximated estimators by $\etazewh$ and $\muzewh$, respectively.
    
    To approximate the local suprema we differ in the degree $k$ of our $\Pcal_k$ finite element (FE) from the eigenvalue problem and solve the local problems using $\Pcalkl$-FE, $\ell\geq 1$.

    We start with $\etazewh$.
    Using the space $\Vfett_z$ and the definition of $\etazew$ in \eqref{eq: comp of etazew defi etazew neu} the following problem can be formulated.
    
    \begin{problem}[$P_{\etazew}$]
        \label{problem: etazew}
    	Find $\wfett \in \Vfett_z \cap \Pcalkl \eqqcolon \Vfettzl$, $\wfett \neq \nullfett$ such that
    	\begin{equation*}
    	   \aetaz(\wfett, \vfett) = Res_z(\vfett) \quad \text{for all } \vfett \in \Vfettzl,
        \end{equation*}
        where
        \begin{equation*}
            \aetaz(\wfett, \vfett) \coloneqq \int\limits_{\omega_z} \phiz \, D\wfett : D\vfett \; dx.
        \end{equation*}
        
    \end{problem}
    
    This problem is well-posed as $Res_z (\einsfett) = Res_h (\phiz) = 0$ holds for constant functions.
    In the following lemma we see that $P_{\etazew}$ coincides with a discretization of \eqref{eq: comp of etazew defi etazew neu}, which yields $\etazewh$.
    
    \begin{lemma}
        \label{lemma: solving problem etazew}
        Problem $P_{\etazew}$ has a unique solution $\wfett \in \Vfettzl$ and $\etazewh = \aetaz(\wfett,\wfett)^{\sfrac{1}{2}}$.
    \end{lemma}
    \begin{proof}[Proof.]
        As in the proof of {\cite[Theorem 5.1]{error_estimator_pou}}, we show that $P_{\etazew}$ has a unique solution. 
        Symmetry of $\aetaz(\cdot,\cdot)$ holds by definition.
        The continuity can be shown using the boundedness of $\phiz$ and $\C$ and a Cauchy-Schwarz inequality.\\ 
        Coercivity of $\aetaz(\cdot, \cdot)$ follows from the equivalence of norms on discrete spaces.
        Thus, we apply Lax Milgram (cf.\ e.g.\ \cite[Theorem 2.7.7]{brenner_scott}) to obtain that $P_{\etazew}$ has a unique solution $\wfett \in \Vfettzl$.\\
        To show the equivalence of $P_{\etazew}$ and \eqref{eq: comp of etazew defi etazew neu}, let $\vfett \in \Vfettzl \setminus \{ \nullfett \}$. 
        Then, using a Cauchy-Schwarz inequality ($\CS$) there holds
        \begin{equation*}
            \frac{Res_z(\vfett)^2}{\aetaz(\vfett,\vfett)} = \frac{\aetaz(\wfett, \vfett)^2}{\aetaz(\vfett,\vfett)} \overset{\CS}{\leq} \frac{\aetaz(\wfett,\wfett) \, \aetaz(\vfett,\vfett)}{\aetaz(\vfett,\vfett)} 
            = \aetaz(\wfett,\wfett).
        \end{equation*} 
        This holds for any $\vfett \in \Vfettzl \setminus \{ \nullfett \}$. In particular, equality holds for $\vfett =\wfett$. Therefore
        \begin{equation*}
            \etazewh = \sup_{\vfett \in \Vfettzl \setminus \{\nullfett\}} \frac{Res_z(\vfett)}{\aetaz(\vfett,\vfett)^{\sfrac{1}{2}}} = \aetaz(\wfett,\wfett)^{\sfrac{1}{2}}.
        \end{equation*}
    \end{proof}
    
    Considering $\muzewh$, we reformulate \eqref{eq: comp of muzew defi muzew neu} to approximate the supremum as a solution of the following problem.
    \begin{problem}[$P_{\muzew}$]
        \label{problem: muzew}
    	Find $\wfett \in \Wfett_z \cap \Pcalkl \eqqcolon \Wfettzl$, $\wfett \neq \nullfett$ such that
    	\begin{equation*}
    	   \amuz(\wfett, \vfett) = Res_z(\vfett) \quad \text{for all } \vfett \in \Wfettzl, 
    	\end{equation*}
    	where 
    	\begin{equation*}
    	     \amuz(\wfett, \vfett) \coloneqq \int\limits_{\omega_z} \phiz \, \varepsilon(\wfett) : \C  \varepsilon(\vfett) \; dx.
    	\end{equation*}
    \end{problem}
    
    To ensure that the problem is well-posed, we need to consider rigid body motions, as they coincide with the kernel of $\varepsilon(\cdot)$.
    Let $RM(\Omega)$ denote the space of rigid body motions on $\Omega$.
    To ensure $Res_z(\rfett) = Res_h (\phiz \, \rfett)=0$ holds for any $\rfett \in RM(\Omega)$, we need our test functions to fulfill $\phiz \, \rfett \in \Hfett_h$.
    Thus we can only consider polynomials of degree $k \geq 2$.
    For the lowest order case we then use $\etalew$.
    The following lemma shows that solving problem $P_{\muzew}$ is equivalent to computing the discretized supremum of
    \eqref{eq: comp of muzew defi muzew neu}.
    
    \begin{lemma}
        \label{lemma: solving problem muzew}
        Assume that for any $\rfett \in RM(\Omega)$ and any hat function $\phiz$ we have that $\rfett \,\phiz \in \Hfett_h$. 
        Then problem $P_{\muzew}$ has a unique solution $\wfett \in \Wfettzl$ and $\muzewh = \amuz(\wfett,\wfett)^{\sfrac{1}{2}}$.
    \end{lemma}
    \begin{proof}[Proof.]
    The proof follows the lines of proof of Lemma~\ref{lemma: solving problem etazew}.
    \end{proof}
    
    Thus we can compute the discrete error estimators $\etazewh$ and $\muzewh$ by solving the local problems $P_{\etazew}$ and $P_{\muzew}$ on local patches, respectively.
 
    This concludes the theoretical analysis of the error estimators. 
    

\section{Numerical experiments}
\label{section: experiments}
    
    In this section we present numerical results of eigenvalues on different domains. 
    Prior to this we look at different forms of patches $\omega_z$ that can occur in the triangulation of the domain.
    We show how to incorporate the constraints in Definitions \ref{def: Vz fuer etazew} and \ref{def: Wz fuer muzew} on the solution spaces $\Vfettzl$ and $\Wfettzl$, respectively.
    
\subsection{Requirements on different patches}
    
    As in Definitions \ref{def: Vz fuer etazew} and \ref{def: Wz fuer muzew}, the solution spaces $\Vfettzl$ and $\Wfettzl$ depend on the location of $z$ and its patch $\omega_z$.
    In particular, we consider the edges that $z$ lies on and use them to differentiate between different types of patches.
    Either $z$ is an interior node or $z$ lies on the boundary.
    If $z$ lies on the boundary, only the two edges on the boundary attached to $z$ inherit the boundary condition of $\partial \Omega$.
    The other edges on $\partial \omega_z$ are left free.
    In general, we have to fix rigid body motions to ensure unicity.
    Therefore, we get the conditions
    \[
        \int_{\omega_z} \vfetth \, dx = \nullfett \quad \text{for } \vfetth \in \Vfettzl,
    \]
    \[
        \int_{\omega_z} \vfetth \, dx = \nullfett \quad \text{and} \quad \int_{\omega_z} \rot \vfetth \, dx = \nullfett \quad \text{for } \vfetth \in \Wfettzl,
   \]
    which can be rewritten using rigid body motions 
    \begin{equation*}
        \rfett_1 = \eoneintwod, \ \rfett_2 = \etwointwod, \ \rfett_3 = \vectorrot \in \RBM (\Omega).
    \end{equation*}
    considering
    \begin{equation*}
         \int_{\omega_z} \vfetth \cdot \rfett_i \, dx = 0,
    \end{equation*}
    with $i=1,2$ for $\etazewh$ and $i=1,2,3$ for $\muzewh$.
    
    Due to the different boundary conditions, we further reduce these conditions depending on the given boundary condition and the position of the patch.
    We differentiate between the same three cases as in Definitions \ref{def: Vz fuer etazew} and \ref{def: Wz fuer muzew}.
    
    \begin{enumerate}[a)]
    \item
    The patch seen in Figure \ref{tikz: patches inner node} lies inside of $\Omega$.
    This means that all edges of this patch have Neumann boundary conditions, i.e.\ $\sigmauh \, \nfett = \nullfett$.
    Hence, we have to consider all conditions for this patch.
    Accordingly, the same holds for patches on the boundary that have edges with Neumann boundary condition.
    
    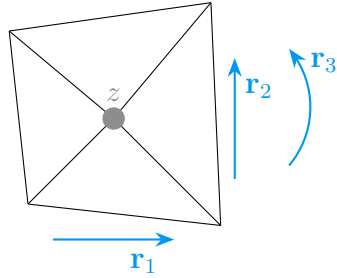
\begin{figure}[H]
    	\centering

\begin{tikzpicture}[rotate = 0, scale = 0.4]
        \draw (50:5cm) -- (140:4.5cm);
        \draw (140:4.5cm) -- (225:4cm);
        \draw (225:4cm) -- (315:5cm);
        \draw (315:5cm) -- (50:5cm);
        \draw (50:0cm) -- (50:5cm);
        \draw (140:0cm) -- (140:4.5cm);
        \draw (225:0cm) -- (225:4cm);
        \draw (315:0cm) -- (315:5cm);
        \draw (0:0cm) node [shape=circle,draw, fill, inner sep = 0pt, minimum size = 8, color = hellgrau] {};
        \node at (90:0.8cm) [color = hellgrau]{$z$};
        
        \draw [-{Stealth[length=0.25cm]}][color = petrol][thick] (-2cm,-4cm) -- (2cm,-4cm);
        \node at (1cm,-4.8cm) [color = petrol] {\large{$\rfett_1$}};
        
        \draw [-{Stealth[length=0.25cm]}][color = petrol][thick] (4cm,-2cm) -- (4cm,2cm);
        \node at (4.8cm,1cm) [color = petrol] {\large{$\rfett_2$}};
        
        \draw [-{Stealth[length=0.25cm]}][color = petrol][thick] (345:6cm) arc (320:400:3cm);
        \node at (375:7.2cm) [color = petrol] {\large{$\rfett_3$}};
        
        \path (180:5cm) -- (180:0cm);
    \end{tikzpicture}
    	\caption{An inner patch $\omega_z$ with $\omega_z \cap \partial \Omega = \emptyset$ or $\omega_z \cap \partial \Omega \subset \GammaN$. The blue arrows show the possible rigid body motions.}
    	\label{tikz: patches inner node}
    \end{figure}
    
    \item
    Regarding Dirichlet boundary conditions we have a fixed value $\ufetth = \nullfett$ on those edges.
    Hence, we do not have to implement any constraints for those patches.
    Note that this holds even if only one of the edges attached to $z$ inherits a Dirichlet boundary condition.
    This case also contains a corner point for gliding boundary conditions.
    There, the gliding boundary condition $\ufetth \cdot \nfett = 0$ has to be fulfilled for the outer normals of both corner edges, thus $\ufetth (z) = \nullfett$.
    As no rigid body motion can be added to $\ufetth$ on this patch, we already obtain unicity of the solution, so we do not need to consider any of the conditions.
    An example of such a patch can be seen in Figure \ref{subfig: patches ecke}.
    \begin{figure}[H]
    	\centering

    \begin{tikzpicture}[rotate = 0, scale = 0.4]
        \draw [color = hellgrau](280:4cm) -- (225:4.5cm);
        \draw [color = hellgrau](225:4.5cm) -- (155:4cm);
        \draw [color = hellgrau](225:0cm) -- (225:4.5cm);
        \draw [thick, color = black][densely dashdotted](155:0cm) -- (155:4cm);
        \draw [thick, color = black][densely dashdotted](280:0cm) -- (280:4cm);
        
        \node (upperedge) at (155:2cm) [shape=circle,  inner sep = 0pt, minimum size = 0, color = blue] {};
        \node (loweredge) at (280:2cm) [shape=circle,  inner sep = 0pt, minimum size = 0, color = blue] {};
        \draw (0:0cm) node [shape=circle,draw, fill, inner sep = 0pt, minimum size = 8, color = gray] {};
        \node at (20:0.8cm) [color = gray] {$z$};
        

    \end{tikzpicture}
        \caption{Patch $\omega_z$ with dashed boundary edges for a polygonal boundary $\omega_z \cap \partial \Omega$.}
        \label{subfig: patches ecke}
    \end{figure}
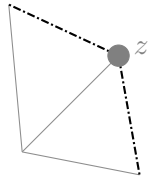

    \item
    Lastly we consider patches with axis parallel gliding boundary edges that are not on a corner.
    
    In Figure \ref{subfig: patch x red} we automatically fulfill $\int_{\omega_z} \ufetth \cdot \rfett_2 \, dx = 0$, as the normal on the boundary is $\nfett = (0, \ 1)^T = \rfett_2$. 
    Therefore we only need to consider $\rfett_1$ in the integral condition for both $\etazewh$ and $\muzewh$, as $\rfett_3$ is eliminated by $\ufetth \cdot \nfett = 0$.
    
    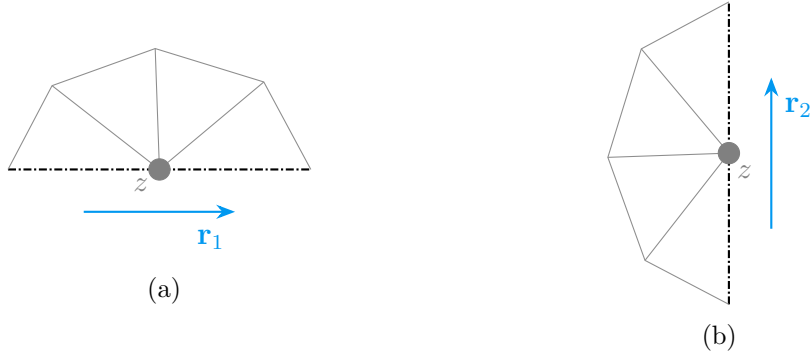
\begin{figure}[H]
    	\centering
    	\begin{subfigure}[h]{0.49\textwidth}
            \centering

    \begin{tikzpicture}[rotate = 0, scale = 0.4]
    
        \draw [thick, color = black][densely dashdotted](180:5cm) -- (0:5cm);
        \draw [color = hellgrau](180:5cm) -- (142:4.5cm) -- (92:4cm) -- (40:4.5cm) -- (0:5cm);
        \draw [color = hellgrau](142:0cm) -- (142:4.5cm);
        \draw [color = hellgrau](92:0cm) -- (92:4cm);
        \draw [color = hellgrau](40:0cm) -- (40:4.5cm);
        
        \draw[-{Stealth[length=0.25cm]}][color = petrol][thick] (-2.5cm,-1.4cm) -- (2.5cm,-1.4cm);
        \node at (1.7cm,-2.3cm) [color = petrol] {\large{$\rfett_1$}};
        \draw (0:0cm) node [shape=circle,draw, fill, inner sep = 0pt, minimum size = 8, color = gray] {};
        \node at (220:0.8cm) [color = gray]{$z$};
    \end{tikzpicture}
            \caption{}
            \label{subfig: patch x red}
        \end{subfigure}
        \begin{subfigure}[h]{0.49\textwidth}
            \centering

    \begin{tikzpicture}[rotate = 90, scale = 0.4]
        \draw [thick, color = black][densely dashdotted](180:5cm) -- (0:5cm);
        \draw [color = hellgrau](180:5cm) -- (142:4.5cm) -- (92:4cm) -- (40:4.5cm) -- (0:5cm);
        \draw [color = hellgrau](142:0cm) -- (142:4.5cm);
        \draw [color = hellgrau](92:0cm) -- (92:4cm);
        \draw [color = hellgrau](40:0cm) -- (40:4.5cm);
        
        \draw [-{Stealth[length=0.25cm]}][color = petrol][thick] (-2.5cm,-1.4cm) -- (2.5cm,-1.4cm);
        \node at (1.5cm,-2.3cm) [color = petrol] {\large{$\rfett_2$}};
        \draw (0:0cm) node [shape=circle,draw, fill, inner sep = 0pt, minimum size = 8, color = gray] {};
        \node at (220:0.8cm) [color = gray]{$z$};
    \end{tikzpicture}
            \caption{}
            \label{subfig: patch y red}
        \end{subfigure}
    	\caption{Patches $\omega_z$ with dashed boundary edges describing $\Ecal_z \cap \Ecal_g$, where the boundary is parallel to (a) the x-axis and (b) the y-axis.
        The blue arrows show the possible rigid body motions in each case.}
    	\label{tikz: patches x y red}
    \end{figure}
    
    Similarly, we do not need to consider $\rfett_1$ and $\rfett_3$ in the integral condition for patches as in Figure \ref{subfig: patch y red}, as the outer normal equals $\nfett = (1,\ 0)^T = \rfett_1$.
    
    \end{enumerate}


In the following subsections we present numerical results on three different domains: the unit square, an L-shaped domain, and Cook's membrane.
Our computations were done using Matlab.


\subsection{Experiments on the unit square}

    We conduct our experiments on the unit square $\Omega = (0,1)^2 \subset \R^2$ and examine all three types of boundary conditions starting with Neumann boundary conditions.
    
    For this experiment we consider finite elements with polynomial degrees $1$ for $\etalew$ and $2$ for $\mulew$, and with varying polynomial degrees for the local problems.
    
    In Figure \ref{fig: square neumann efficiency p1 p2 (+0,1,2,3)} we see that the efficiency index already stabilises for $\Pcal_{1+1}$- and $\Pcal_{2+1}$-FE, respectively. 
    As just one more polynomial degree for the local problems suffices, we will use $\Pcal_{k+1}$-FE in the following experiments (if not stated otherwise).

    \begin{figure}[tbp]
        \centering
        \begin{subfigure}[b]{0.49\textwidth}
            \centering
            \includegraphics[width=1\textwidth]{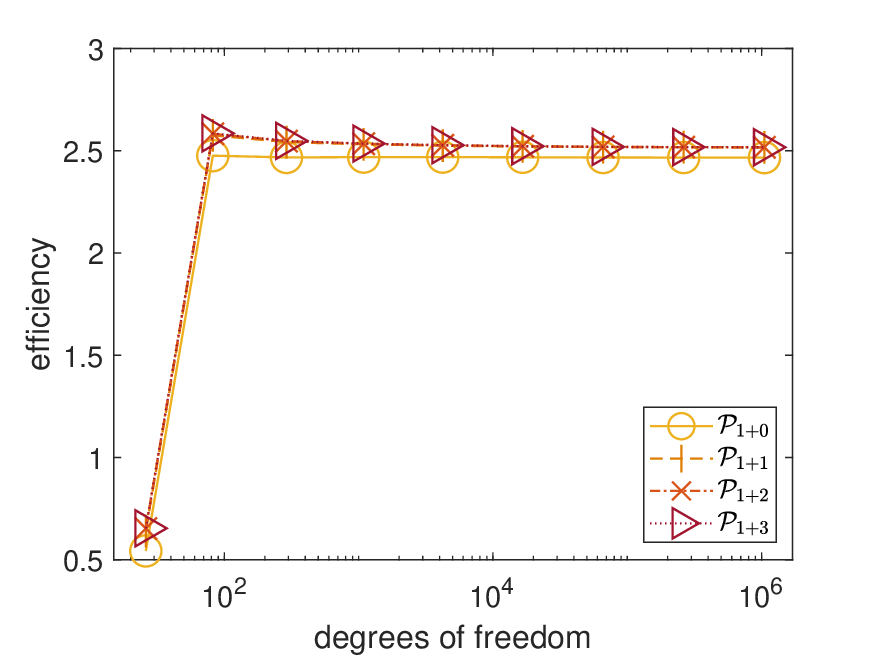}
            \caption{}
            \label{subfig: square neumann eff p1}
        \end{subfigure}
        \begin{subfigure}[b]{0.49\textwidth}
            \centering
            \includegraphics[width=1\textwidth]{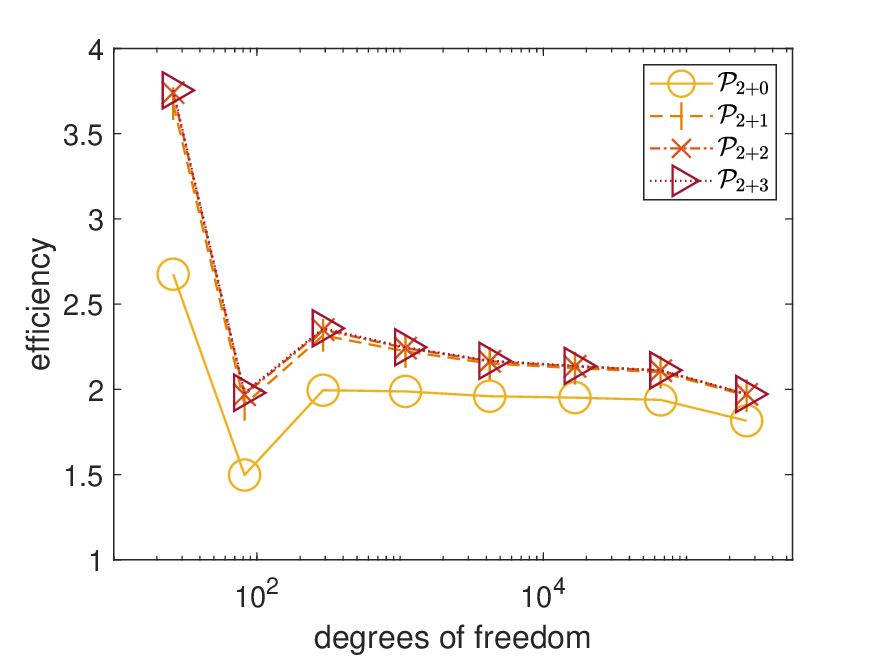}
            \caption{}
            \label{subfig: square neumann eff p2}
        \end{subfigure}
        \caption{Efficiency index of (a) $\etalew^2$ using $\Pcal_1$-FE and (b) $\mulew^2$ using $\Pcal_2$-FE with different local polynomial degrees on the unit square with Neumann boundary conditions.}
        \label{fig: square neumann efficiency p1 p2 (+0,1,2,3)}
    \end{figure}

       \begin{figure}[tbp]
        \centering
        \includegraphics[width=0.9\textwidth]{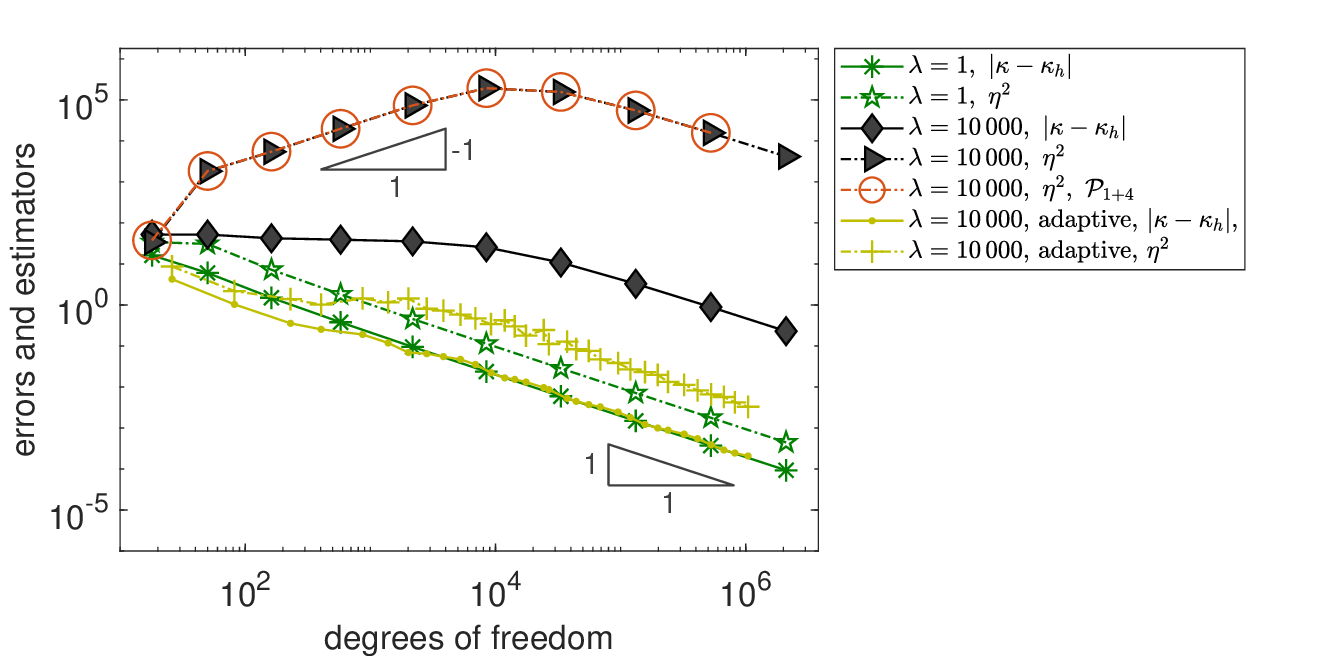}
        \caption{Locking on the square with gliding boundary conditions for $\lambda =1 $ and $\lambda = 10\,000$ using $\Pcal_1$-FE for adaptive and uniform criss meshes.}
        \label{fig: square sb locking p1}
    \end{figure}
    \begin{figure}[tbp]
        \centering
        \includegraphics[width=0.9\textwidth]{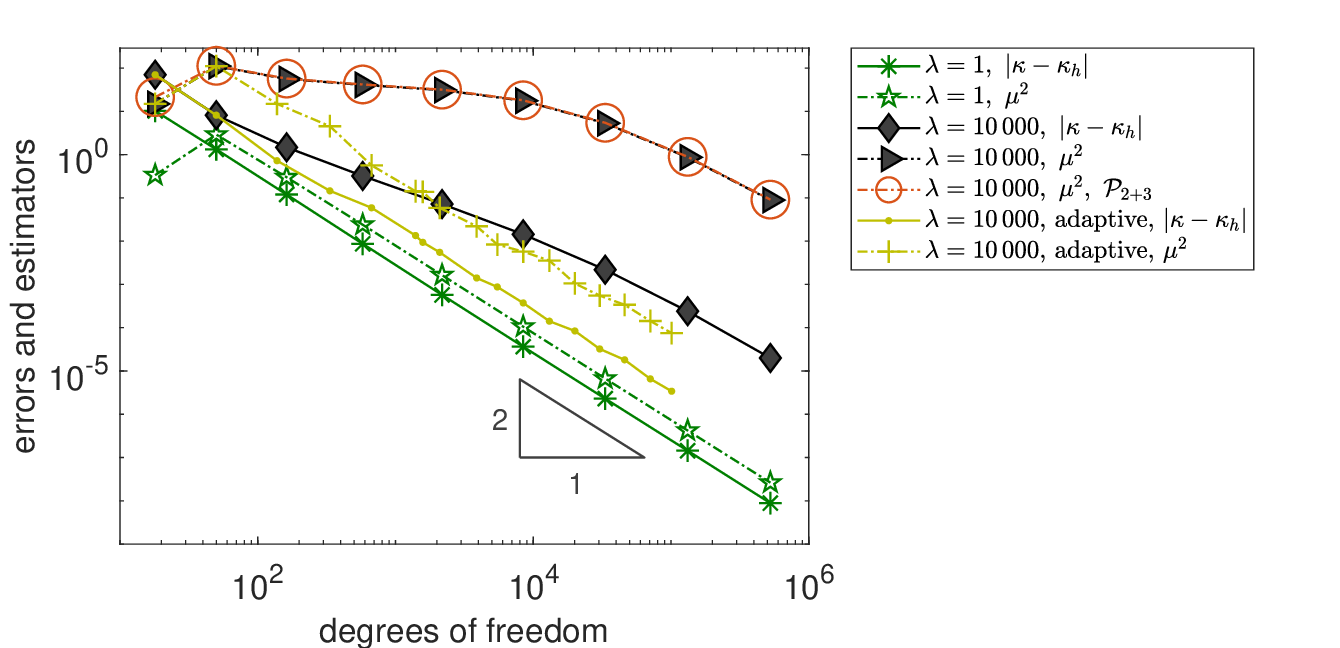}
        \caption{Locking on the square with gliding boundary conditions for $\lambda =1 $ and $\lambda = 10\,000$ using $\Pcal_2$-FE for adaptive and uniform criss meshes.}
        \label{fig: square sb locking p2}
    \end{figure}

    We also consider a locking benchmark on the square. 
    For squares with gliding boundary conditions the eigenpairs are given as s-modes and p-modes in \cite[Chapter 2.3]{jones_modes}.
    
    The s-modes and eigenvalues are independent of the critical Lamé parameter $\lambda$.
    Applied to $\Omega=(0,1)^2$ they are defined as
    \begin{align*}
        \ufett_s((x,y)^T) &\coloneqq \Big( n \,\sin \big( m \pi x\big)\, \cos \big( n \pi y \big),\ -m \,\cos \big( m \pi x \big)\, \sin \big( n \pi y \big) \Big)^T, \\
        \kappa_s &\coloneqq \mu \pi^2 \big( m^2 + n^2 \big), \ m, n = 1,2,\dots.
    \end{align*}
    
    The first eigenvalue is an s-eigenvalue given by
    \begin{equation*}
        \kappa_1 =2 \mu \pi^2,
    \end{equation*}
    where $m=n=1$.
    Due to its independence, we can increase the Lamé parameter $\lambda$ without affecting $(\kappa_1, \ufett_1)$ to demonstrate the locking phenomenon.

    The results using different grid structures for both $\Pcal_1$- and $\Pcal_2$-FE are shown in Figures \ref{fig: square sb locking p1} and \ref{fig: square sb locking p2}, respectively.
    The convergence rates are plotted with respect to the degrees of freedom.
    Note that in two dimensions on uniformly refined grids the relation $h \approx N^{-1/2}$ holds for the mesh size $h$ and the degrees of freedom $N$.
    
    Both figures show that on structured criss meshes locking occurs both in the error and in the estimators.
    Note that the estimators worsen their preasymptotic convergence rates by one degree more compared to the eigenvalue error.
    This is however not due to locking of the local problems, as we see in both figures that using $\Pcal_{5}$-FE for the local problems for $\etazewh$ and $\muzewh$, respectively, does not change this behavior.
    
    Using adaptive mesh refinement that breaks the criss structure also breaks the locking behavior, similar to the known effect on criss-cross meshes.
    
    Lastly we look at higher eigenvalues for Dirichlet boundary conditions.
    In particular, we computed the eigenvalues with indices $m=1,\dots,4,12,13$, which are simple eigenvalues.
    In Figure \ref{fig: square dirichlet several eigenvalues} we see that computing the eigenvalues $\kappa_1,\kappa_2$ and $\kappa_4$ does not yield the optimal convergence rates as the corresponding eigenfunctions are not smooth.
    Note that the absolute eigenvalue error increases for higher eigenvalues.\\
    The von Mises stresses of $\ufett_{12}$ and $\ufett_{13}$ can be seen in Figure \ref{fig: square dirichlet von mises}.
    
    \begin{figure}[tbp]
        \centering
        \includegraphics[width=0.7\textwidth]{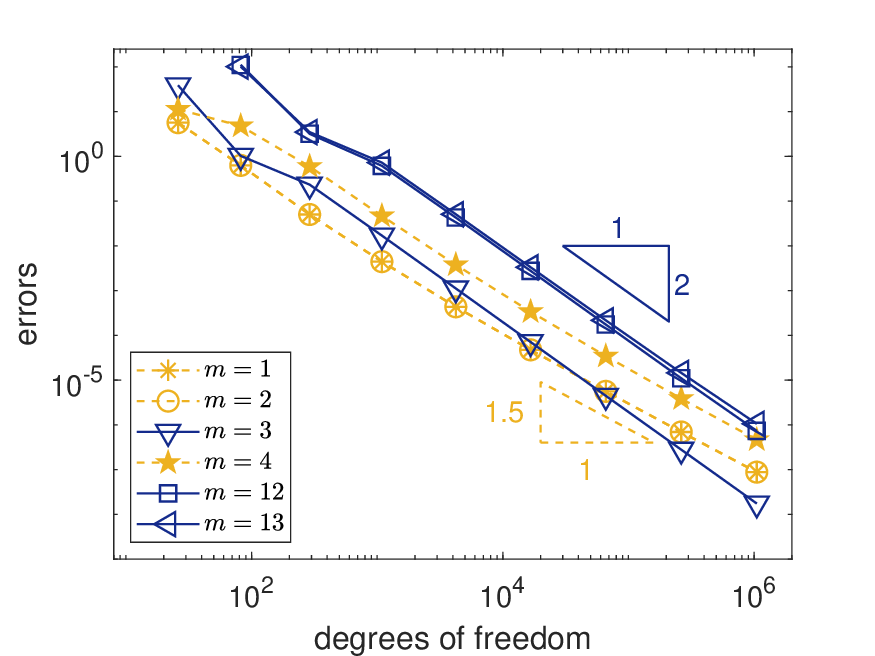}
        \caption{Error for different eigenvalues $\kappa_m$, $m=1,\dots,4,12,13$ on the unit square with Dirichlet boundary using $\Pcal_2$-FE.}
        \label{fig: square dirichlet several eigenvalues}
    \end{figure}
    
    \begin{figure}[tbp]
        \centering
        \begin{subfigure}[b]{0.49\textwidth}
            \centering
            \includegraphics[width=1\textwidth]{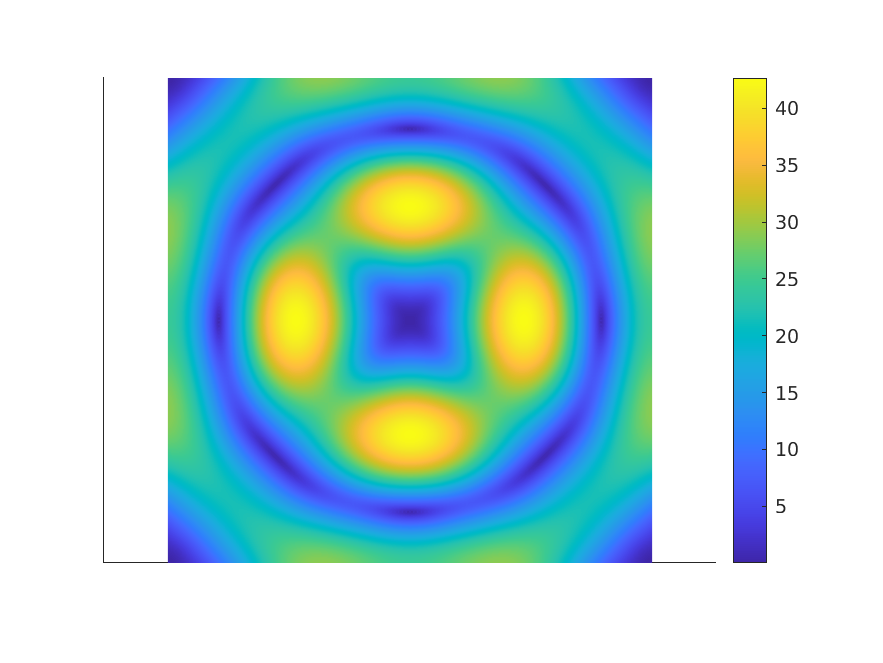}
            \caption{$m=12$}
            \label{subfig: square dirichlet von mises m=12}
        \end{subfigure}
        \begin{subfigure}[b]{0.49\textwidth}
            \centering
            \includegraphics[width=1\textwidth]{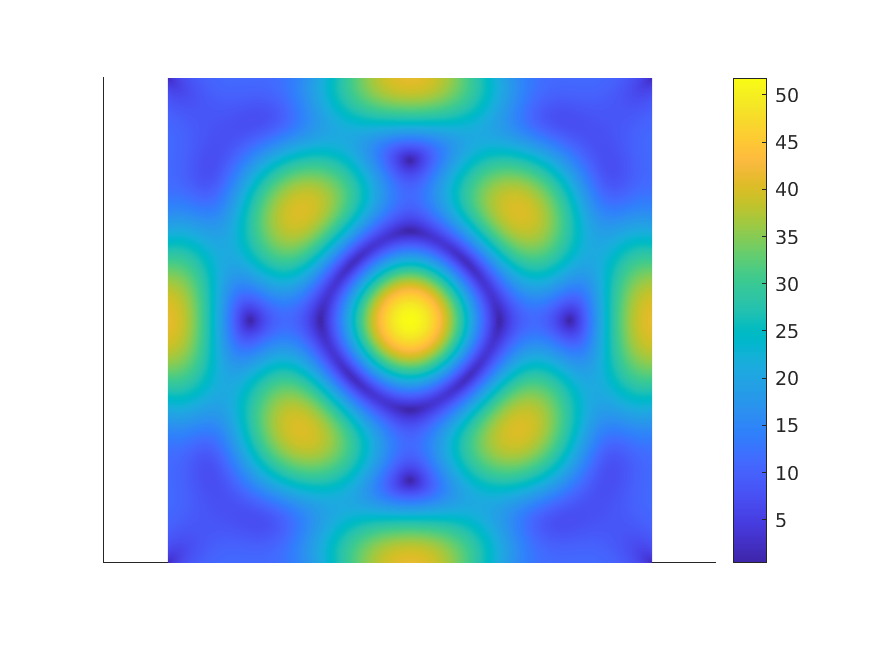}
            \caption{$m=13$}
            \label{subfig: square dirichlet von mises m=13}
        \end{subfigure}
        
        \caption{Von Mises stress for the unit square with Dirichlet boundary for different eigenfunctions $\ufett_m$ computed with $\Pcal_2$-FE.}
        \label{fig: square dirichlet von mises}
    \end{figure}


\subsection{Experiments on an L-shaped domain}
    
    In this subsection we compare uniform and adaptive mesh refinement as well as the estimator $\mulew^2$ to a residual estimator \cite{cg_2011} for Dirichlet and Neumann boundary conditions
    \begin{equation*}
        \etares^2 \coloneqq \sum_{T \in \Tcal} h_T^2 \, \|\kappah \ufetth + \divsigmauh\|^2_{0,T} + \sum_{E \in \Ecal} h_E \, \|[\sigmauh]\; \nfett_E\|^2_{0,E},
    \end{equation*}
    on a domain that contains a re-entrant corner, in particular an L-shaped domain.

    The L-shaped domain with Neumann boundary conditions used for the computations is given by $\Omega = [-1,1]^2 \setminus (0,1)\times(-1,0) \subset \R^2$.
    Due to the Neumann boundary conditions, the first three eigenvalues corresponding to the rigid body motions are equal to zero.
    Thus, we compute the smallest nonzero eigenvalue $\kappa_4$.
    
    Figure \ref{subfig: lshape neumann plot} shows that for $\Pcal_3$-FE the optimal convergence rate is attained for adaptive refined meshes, but not for uniformly refined ones.
    A comparison of the efficiency index of $\etares^2$ and $\mulew^2$ in Figure \ref{subfig: lshape neumann efficiency} shows that the efficiency constant of $\mulew^2$ is closer to 1 by two orders of magnitude.
    
    \begin{figure}[tbp]
        \centering
        \begin{subfigure}[b]{0.49\textwidth}
            \centering
            \includegraphics[width=1\textwidth]{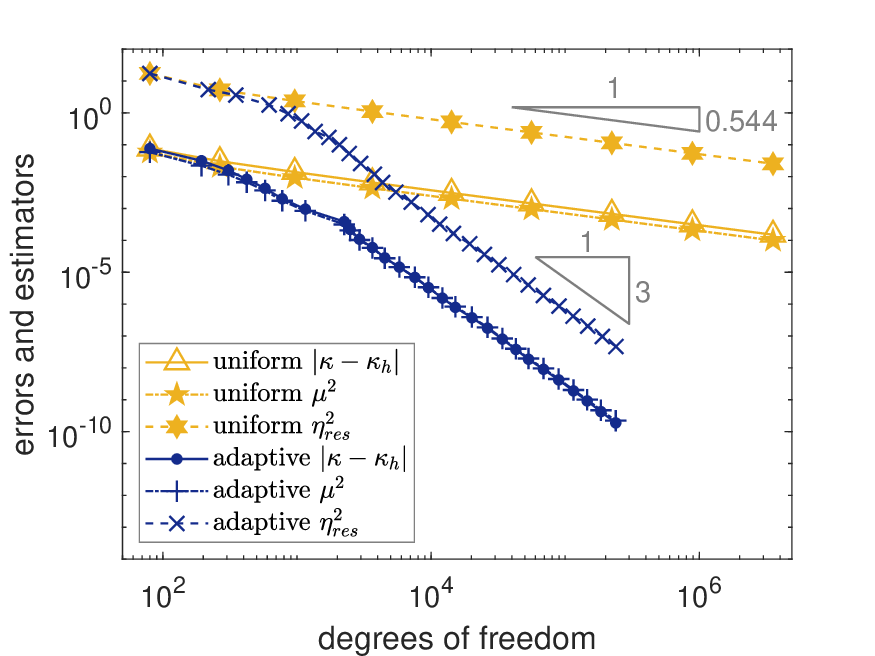}
            \caption{}
            \label{subfig: lshape neumann plot}
        \end{subfigure}
        \begin{subfigure}[b]{0.49\textwidth}
            \centering
            \includegraphics[width=1\textwidth]{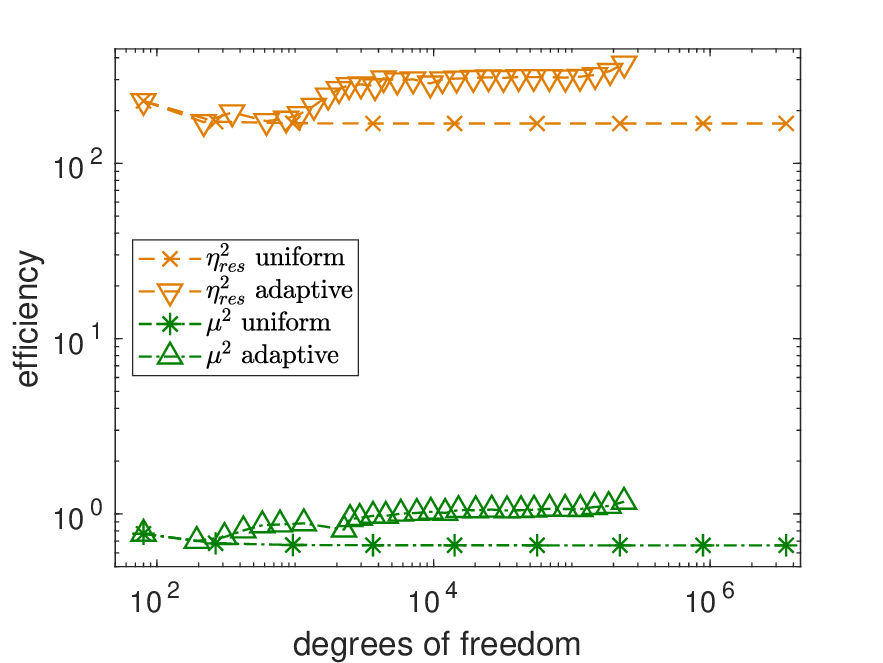}
            \caption{}
            \label{subfig: lshape neumann efficiency}
        \end{subfigure}
        \caption{(a) Error and estimator plot and (b) efficiency index of $\mulew^2$ compared to $\eta_{res}^2$ for the L-shaped domain with Neumann boundary using $\Pcal_3$-FE for the eigenvalue with index $m = 4$ with uniform and adaptive mesh refinement.}
        \label{fig: lshape neumann eff plot}
    \end{figure}

\subsection{Experiments on Cook's membrane}
 
    To show that mixed boundary conditions can be approximated by the estimator, we consider Cook's membrane with Dirichlet and Neumann boundary parts as shown in Figure \ref{fig: Cooks tikz domain} and compute the smallest eigenvalue $\kappa_1$. 
    
    Again we compare uniform and adaptive mesh refinement as well as the estimator $\mulew^2$ to the residual estimator $\etares^2$.
    
    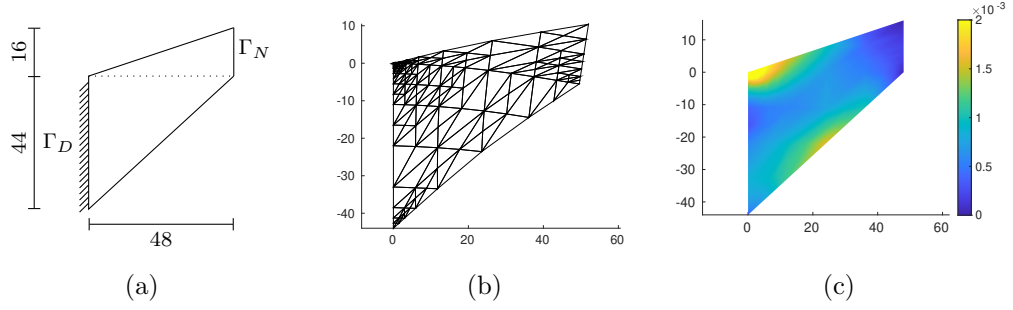
\begin{figure}[tbp]
    	\centering
    	\begin{subfigure}[b]{0.3\textwidth}
    		\centering
\begin{tikzpicture}[scale = 0.04]
    \path (0,0) coordinate (A);
    \path (0,44) coordinate (B);
    \path (48,60) coordinate (C);
    \path (48,44) coordinate (D);
    
    \draw (D) -- (A) -- (B) -- (C);
    \draw[dotted] (B) -- (D);
    
    \foreach \y in {1,...,21}
    {
       \draw (0,\y*2) -- ++(-3,-3);
    }
    \draw (48,44) -- (48,60);
    
    \draw (-11,22) node[fill=white] {\scriptsize $\Gamma_D\!\!$};
    \draw (55,53) node {\scriptsize $\Gamma_N$};
    
    \draw[|-|] (0,-5) -- (48,-5);
    \draw (24,-10) node {\scriptsize 48};
    
    \draw[|-|] (-18,0) -- (-18,44);
    \draw[-|] (-18,44) -- (-18,60);
    \draw (-23,22) node[rotate=90] {\scriptsize 44};
    \draw (-23,52) node[rotate=90] {\scriptsize 16};

\end{tikzpicture}%
    		\caption{}
    	\end{subfigure}
    	\begin{subfigure}[b]{0.3\textwidth}
    		\centering
    		\includegraphics[width=1\textwidth]{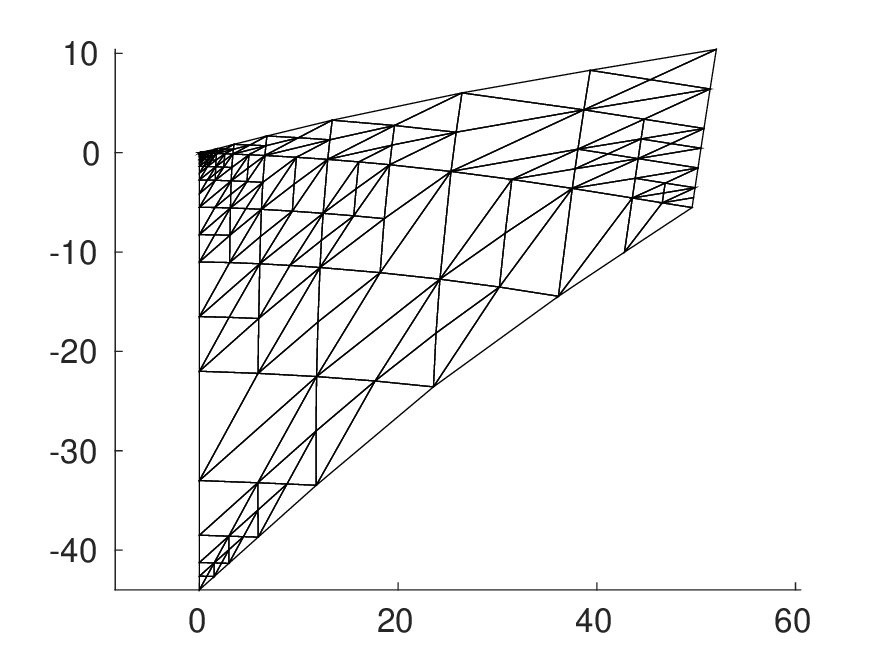}
    		\caption{}
    	\end{subfigure}
    	\begin{subfigure}[b]{0.32\textwidth}
    		\centering
    		\includegraphics[width=1\textwidth]{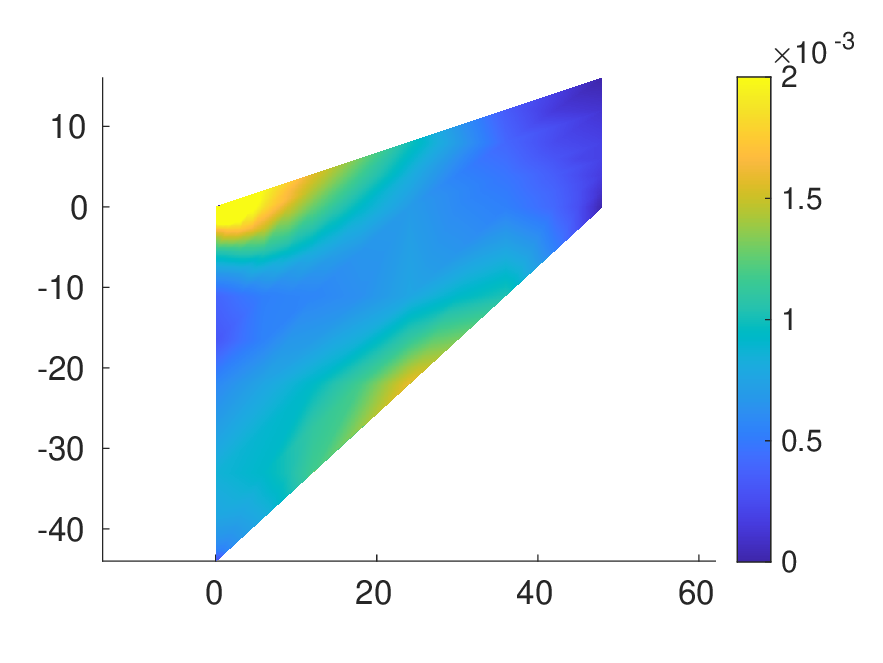}
    		\caption{}
    	\end{subfigure}
    	\caption{(a) Cooks membrane with Dirichlet boundary part $\GammaD$ and Neumann boundary part $\GammaN = \partial \Omega \setminus \GammaD$ with (b) displacement and (c) von Mises stress for $m=1$ using $\Pcal_4$-FE with adaptive mesh refinement. 
    	}
    	\label{fig: Cooks tikz domain}
    \end{figure}
    
    In Figure \ref{fig: cooks plot and eff}, we observe optimal convergence rates for adaptive, but not for uniform mesh refinement.
    As for the L-shaped domain, the efficiency index of $\mulew^2$ is closer to 1 compared to the efficiency of $\etares^2$ by two orders of magnitude.
    
    \begin{figure}[tbp]
        \centering
        \begin{subfigure}[b]{0.49\textwidth}
            \centering
            \includegraphics[width=1\textwidth]{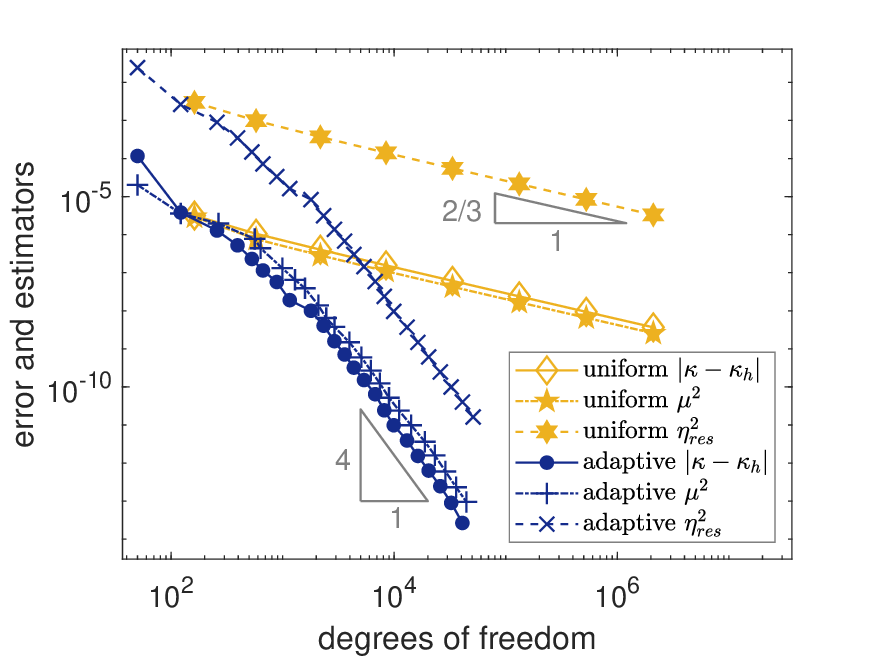}
            \label{subfig: cooks plot}
            \caption{}
        \end{subfigure}
        \begin{subfigure}[b]{0.49\textwidth}
            \centering
            \includegraphics[width=1\textwidth]{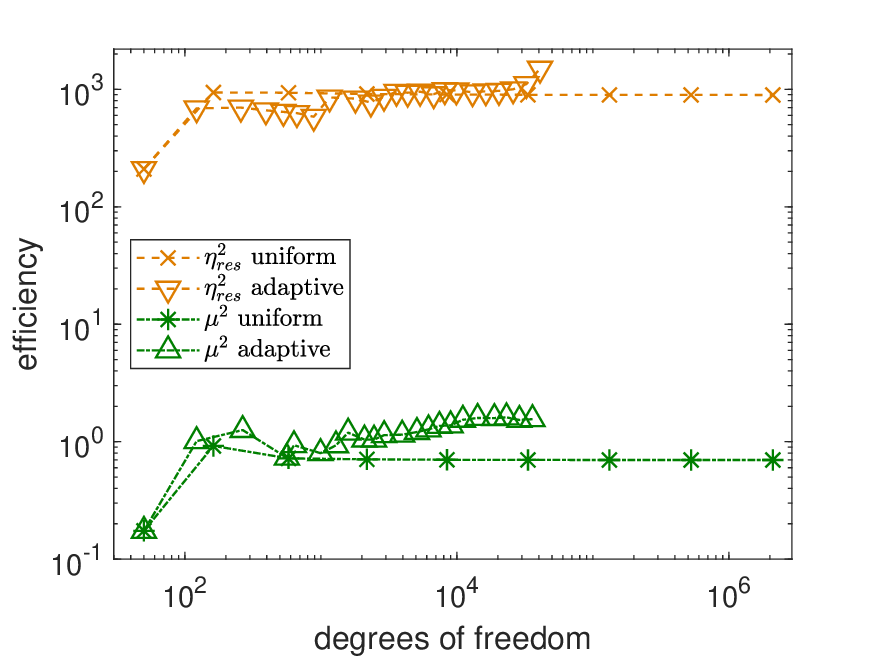}
            \label{subfig: cooks efficiency}
            \caption{}
        \end{subfigure}
        \caption{(a) Error and estimator plot and (b) efficiency index of $\mulew^2$ compared to $\eta_{res}^2$ for Cooks membrane with Dirichlet and Neumann boundary parts using $\Pcal_4$-FE with uniform versus adaptive refinement for $\kappa_1$.}
        \label{fig: cooks plot and eff}
    \end{figure}


\bibliographystyle{plain}
\bibliography{Referenzen}

\end{document}